\documentclass[a4paper,11pt]{article}
\usepackage{color}
\usepackage{verbatim}
\usepackage{theorem,amsmath,amssymb,amscd}
\usepackage{booktabs}
\usepackage[hyperfootnotes=false, colorlinks, linkcolor={blue}, citecolor={magenta}, filecolor={blue}, urlcolor={blue}]{hyperref}

\theoremstyle{change}
\allowdisplaybreaks
\newcommand{\Hl}{{\mathbb H}}
\newcommand{\C}{{\mathbb C}}
\newcommand{\R}{{\mathbb R}}
\newcommand{\Q}{{\mathbb Q}}
\newcommand{\Z}{{\mathbb Z}}

\newcommand{\g}{{\frak g}}
\newcommand{\p}{{\mathfrak p}}

\newcommand{\n}{{\mathfrak n}}

\newcommand{\la}{{\mathfrak a}}
\newcommand{\mk}{{\mathfrak k}}

\newcommand{\mO}{{\cal O}}
\newcommand{\A}{{\mathbb A}}

\newcommand{\real}{\operatorname{Re}}

\newcommand{\tr}{\operatorname{tr}}

\newcommand{\IM}{\operatorname{Im}}

\newcommand{\cG}{{\cal G}}

\newcommand{\cM}{{\cal M}}

\newcommand{\ch}{\operatorname{char}}

\newcommand{\Hp}{{\cal H}_p}

\newcommand{\GL}{{\rm GL}}

\newcommand{\SL}{{\rm SL}}

\newcommand{\Symp}{{\rm Sp}}
\newcommand{\GSp}{{\rm GSp}}
\newcommand{\PGSp}{{\rm PGSp}}

\newcommand{\SSp}{{\rm Sp}}

\newcommand{\mat}[4]{{\setlength{\arraycolsep}{0.5mm}\left[
\begin{array}{cc}#1&#2\\#3&#4\end{array}\right]}}
\newcommand{\qed}{\hspace*{\fill}\rule{1ex}{1ex}}

\numberwithin{equation}{section}
\theoremstyle{plain}
 \newtheorem{thm}{Theorem}[section]
 \newtheorem{prop}[thm]{Proposition}
 \newtheorem{lem}[thm]{Lemma}
 
 \newtheorem{conj}[thm]{Conjecture}

\theoremstyle{definition}
 \newtheorem{defn}[thm]{Definition}
 \newtheorem{rem}[thm]{Remark}

 \newenvironment{proof}{\vspace{1ex}\noindent{\it Proof.}\hspace{0.1em}}
	{\hfill\qed\vspace{2ex}}
\begin{document}
\title{Lifting to $\GL(2)$ over a division quaternion algebra and an explicit construction of CAP representations}
\author{Masanori Muto, Hiro-aki Narita\footnote{Partly supported by Grant-in-Aid for Scientific Research (C) 24540025, Japan Society for the Promotion of Science.}  and Ameya Pitale\footnote{Partly supported by National Science Foundation grant DMS-1100541.}}
\date{}
\maketitle
\begin{abstract}
The aim of this paper is to carry out an explicit construction of CAP representations of $\GL(2)$ over a division quaternion algebra with discriminant two. We first construct cusp forms on such group explicitly by lifting from Maass cusp forms for the congruence subgroup $\Gamma_0(2)$. We show that this lifting is non-zero and Hecke-equivariant.  This allows us to determine each local component of such a cuspidal representation. We then know that our cuspidal representations provide examples of CAP representations, and in fact, counterexamples of the Generalized Ramanujan conjecture.
\end{abstract}
\section{Introduction}
One of the fundamental problems in the theory of automorphic forms or representations is to study the Ramanujan conjecture. To review it let $\cG$ be a reductive algebraic group over a number field $F$ and let $\A:=\otimes'_{v\le\infty}F_v$ be the ring of adeles for $F$, where $F_v$ denotes the local field at a place $v$. 
\begin{conj}
Let $\pi=\otimes'_{v\le\infty}\pi_v$ be an irreducible cuspidal representation of $\cG(\A)$, where $\pi_v$ denotes the local component of $\pi$ at a place $v$. 
Then $\pi_v$ is tempered for every $v\le\infty$.
\end{conj}
Nowadays it is widely known that counterexamples of this conjecture are found.  A well-known example is given by the Saito-Kurokawa lifting to holomorphic Siegel cusp forms of degree two~(cf.~\cite{Ku}). As another well-known example there is the work by Howe and Piatetski-Shapiro \cite{HP}, which gives a counterexample by a theta lifting from ${\rm O}(2)$ to $\SSp(2)$. 
In fact, it is expected that liftings from automorphic forms on a smaller group provide such counterexamples. 
On the other hand, let us recall that there is the notion of CAP representation (Cuspidal representation Associated to a Parabolic subgroup), which has been originally introduced by Piatetski-Shapiro \cite{Ps}. This is a representation theoretic approach to find counterexamples of the Ramanujan conjecture. 

We now note that the Ramanujan conjecture for the general linear group $\GL(n)$ is strongly believed. In fact, by Jacquet and Shalika \cite{JS}, it can be shown that the CAP phenomenon never occurs for $\GL(n)$. More generally, it is expected that the conjecture would hold for generic cuspidal representations of quasi-split reductive groups. 
In view of the Langlands functoriality principle for quasi-split groups and their inner forms, the Ramanujan conjecture for the inner forms are quite natural and interesting to study. 
We can thereby say that CAP representations of the cases of the inner forms are significant to discuss.  
To define the notion of CAP representations for the inner forms we follow Gan \cite{G1} and Pitale \cite{P}~(cf.~Definition \ref{CAP-def}). 
Let us note that the Saito-Kurokawa lifting deals with the case of the split symplectic group $\GSp(4)$ of degree two. The case of the non-split inner form $\GSp(1,1)\simeq {\rm GSpin}(1,4)$ of $\GSp(4)$ is considered by \cite{G1} and \cite{P} etc. 

In this paper we take up the case of $\GL_2(B)$ over the division quaternion algebra $B$ with discriminant two, where note that $\GL_2(B)$ is an inner form of the split group $\GL(4)$. 
 Our results provide an explicit construction of cusp forms on $\GL_2(B)$ with lifts from Maass cusp forms for the congruence subgroup $\Gamma_0(2)$, and show that the cuspidal automorphic representations generated by such lifts are CAP representations of $\GL_2(B)$. The method of our construction of the lifting is what follows \cite{P}, which deals with an explicit construction of lifting to $\GSp(1,1)$. In fact, note that the cusp forms constructed by our lifting are viewed as Maass cusp forms on the 5-dimensional real  hyperbolic space, while the lifting considered in \cite{P} provides Maass cusp forms on the 4-dimensional real hyperbolic space. As in \cite{P}, to prove the automorphy of our lifts, we use the converse theorem \cite{Ma} by Maass, which is useful for real hyperbolic spaces of arbitrary dimension. 

We explain the explicit construction of our lifting. Let $f$ be a Maass cusp form for $\Gamma_0(2)$ which is an eigenfunction of the Atkin-Lehner involution. 
Let $\{c(n)\}_{n\in\Z\setminus\{0\}}$ be Fourier coefficients of $f$. From $c(n)$'s we define numbers $A(\beta)$'s~(cf.~(\ref{Abeta-defn})) for $\beta\in B\setminus\{0\}$ in order to construct our lifting  to a cusp form $F_f$ on $\GL_2(\Hl)$ in the non-adelic setting. Actually $A(\beta)$'s are nothing but Fourier coefficients of $F_f$. 
The statement of our first result is as follows:~(cf.~Theorem \ref{lift-thm})
\begin{thm}
Let $f$ be a non-zero Mass cusp form which is an eigenfunction of the Atkin-Lehner involution.  Then $F_f$ is a non-zero cusp form on $\GL_2(\Hl)$.
\end{thm}

Another result is that the cuspidal representations generated by $F_f$'s are CAP representations of $\GL_2(B)$ and provide counterexamples of the Generalized Ramanujan conjecture. To be more precise, assume that $f$ is a Hecke eigenform. We can regard $F_f$ as a cusp form on the adele group $\cG(\A)$ with $\cG=\GL_2(B)$. We can show that $F_f$ is a Hecke eigenform~(cf.~Section \ref{Hecke-action}). Then the strong multiplicity one theorem proved by Badulescu and Renard \cite{Bad},~\cite{Bad-R} implies that $F_f$ generates an irreducible cuspidal representation $\pi:=\otimes'_{p\le\infty}\pi_p$ of $\cG(\A)$. 
By our detailed study on Hecke eigenvalues of $F_f$ we can determine local representations $\pi_p$ for every $p<\infty$. We can also determine $\pi_p$ explicitly at $p=\infty$ by the calculation of the eigenvalue for the Casimir operator. For this we note that every $\pi_p$ is unramified (at an odd prime) or spherical (at $p=2$ or $\infty$). We can show that $\pi_p$~(respectively $\pi_{\infty}$) is non-tempered at every odd prime $p$~(respectively tempered at $p=\infty$). 
If we further assume that $f$ is a new form, we can also show the non-temperedness of $\pi_p$ at $p=2$.  These lead to our another result as follows~(cf.~Theorem \ref{CAP-thm},~Theorem \ref{Counter-eg-RC}):
\begin{thm}
(1)~Let $f$ be a non-zero Hecke eigen  cusp form and $F_f$ be the lift. Let $\sigma_f$ and $\pi_F$ be  irreducible cuspidal representations generated by $f$ and $F=F_f$ respectively. Then $\pi_F$ is nearly equivalent to an irreducible component of ${\rm Ind}_{P_2(\A)}^{\GL_4(\A)}(|{\rm det}|^{-1/2}\sigma \times |{\rm det}|^{1/2}\sigma)$. Here $P_2$ is the standard parabolic subgroup of $\GL_4$ with Levi subgroup $\GL_2 \times \GL_2$. Namely $\pi_F$ is a CAP representation.\\
(2)~The cuspidal representations $\pi_F$'s are counterexamples of the Ramanujan conjecture.
\end{thm}

Let $\pi'$ be the unique irreducible quotient of ${\rm Ind}_{P_2(\A)}^{\GL_4(\A)}(|{\rm det}|^{-1/2}\sigma \times |{\rm det}|^{1/2}\sigma)$. This is denoted by ${\rm MW}(\sigma, 2)$ in Section 18 of \cite{Bad-R}, which is a non-cuspidal, discrete series representation of $\GL_4(\A)$.  Since $\sigma$ is not the image of a cuspidal representation of $B_{\A}^\times$ under the Jacquet-Langlands correspondence, $\pi'$ is $B$-compatible according to Proposition 18.2, part (a) of \cite{Bad-R}. Hence there exists a discrete series representation $\pi$ of $\GL_2(B_{\A})$ which maps to $\pi'$ under the Jacquet-Langlands correspondence. Also, from Proposition 18.2, part (b) of \cite{Bad-R}, the representation $\pi$ has to be cuspidal. By the strong multiplicity one theorem for $\GL_2(B)$, the representation $\pi$ has to be exactly the same as $\pi_F$ obtained from the classical construction. The novelty of our method is that we obtain explicit formula for the lift in terms of Fourier expansions which are valid for non-Hecke eigenforms as well. In addition, the classical method immediately shows that the lifting is a linear non-zero map. 

Let us remark that Grobner \cite{G2} has also obtained examples of CAP representations for $\GL_2(B)$ using the results of \cite{Bad-R}. 
The example by Grobner \cite{G2} has a non-tempered local component at the archimedean place, while our cuspidal representation $\pi$ has a tempered local component at the place as is remarked above. 

The outline of this paper is as follows. In Section 2 we first introduce basic notation of algebraic groups and Lie groups. Then we next introduce automorphic forms in our concern. In Section 3 we study a zeta-integral attached to a Maass cusp form $f$, which is necessary to use the converse theorem by Maass later. Then the explicit construction of cusp forms on $\GL_2(\Hl)$ is given by lifts $F_f$'s in Section 4. In Section  5 we view $F_f$ as a cusp form on the adele group $\cG(\A)$ and prove that $F_f$ is a Hecke eigenform at every finite place. Then, in Section 6, we determine local components $\pi_p$ of the cuspidal representation $\pi$ generated by $F_f$ for all places $p\le\infty$. We thus see that $\pi$ is a CAP representation and provides a counterexample of the Ramanujan conjecture.
\subsection*{Acknowledgement}
From Masao Tsuzuki we have known Weyl's law of new Maass cusp forms for $\Gamma_0(2)$, which leads to the existence of a non-zero cuspidal representation $\pi_F=\otimes'_{p\le\infty}\pi_p$ with non-tempered local components $\pi_p$ at every $p<\infty$. Marko Tadi{\'c} kindly made his remark on the proof of the non-temperedness of $\pi_2$. 
These are summarized as Remark \ref{local-rep-rem}. Our profound gratitude is due to them. 
We would like to thank A. Raghuram for informing us of the paper \cite{G2} by Grobner. We would also like to thank Abhishek Saha for several discussions leading to the proof of the non-vanishing of the lifting.
\section{Basic notations}
\subsection{Algebraic group, real Lie groups and the 5-dimensional hyperbolic space}\label{gps-hypsp}
Let $B$ be the definite quaternion algebra over $\Q$ with discriminant $d_B=2$. The algebra $B$ is given by $B=\Q+\Q i+\Q j+\Q k$ with a basis $\{1,~i,~j,~k\}$ characterized by the conditions
\[
i^2=j^2=k^2=-1,~ij=-ji=k.
\]
Let $\cG$ be the $\Q$-algebraic group  defined by its group of $\Q$-rational points
\[
\cG(\Q)=\GL_2(B).
\]
Here $\GL_2(B)$ is the general linear group over $B$, which consists of elements in $M_2(B)$ whose reduced norms are non-zero. 
Let $\Hl=B\otimes_{\Q}\R$, which is nothing but the Hamilton quaternion algebra $\R+\R i+\R j+\R k$.  
Let $\Hl\ni x\mapsto\bar{x}\in\Hl$ denote the main involution of $\Hl$, and $\tr(x)=x+\bar{x}$ and $\nu(x):=x\bar{x}$ be the reduced trace and the reduced norm of $x\in\Hl$ respectively. In what follows, we often use the notation $|\beta|:=\sqrt{\nu(\beta)}$ for $\beta\in\Hl$. We put $\Hl^-:=\{x\in\Hl\mid\tr(x)=0\}$ to be the set of pure quaternions, and $\Hl^1:=\{x\in\Hl\mid\nu(x)=1\}$.

Denote by $G:=\GL_2(\Hl)$ the general linear group of degree two with coefficients in the Hamilton quaternion algebra $\Hl$. The Lie group $G$ admits an Iwasawa decomposition 
\[
G=Z^+NAK,
\]
where
\begin{align}\label{Iwasawa-decomp}
Z^+&:=\left\{\left.
\begin{bmatrix}
c & 0\\
0 & c
\end{bmatrix}~\right|~c\in\R^{\times}_+\right\},\quad
N:=\left\{\left.n(x)=
\begin{bmatrix}
1 & x\\
0 & 1
\end{bmatrix}~\right|~x\in\Hl\right\},\\ \nonumber
A&:=\left\{\left.
a_y:=
\begin{bmatrix}
\sqrt{y} & 0\\
0 & \sqrt{y}^{-1}
\end{bmatrix}~\right|~y\in\R^{\times}_+\right\},\quad
K:=\{k\in G\mid {}^t\bar{k}k=1_2\}.
\end{align}
The subgroup $Z^+$ is contained in the center of $G$ and $K$ is a maximal compact subgroup of $G$, which is isomorphic to the definite symplectic group $\Symp^*(2)$. 

Let us consider the quotient $G/Z^+K$, which is realized as
\[
\left\{\left.
\begin{bmatrix}
y & x\\
0 & 1
\end{bmatrix}\right|~y\in\R^{\times}_+,~x\in\Hl\right\}.
\]
This gives a realization of the 5-dimensional hyperbolic space.
\subsection{Lie algebras}\label{Lie-alg}
The Lie algebra $\g$ of $G$ is nothing but $M_2(\Hl)$, and has an Iwasawa decomposition
\[
\g={\frak z}\oplus\n\oplus\la\oplus\mk,
\]
where
\begin{align}\label{Iwasawa-decomp-alg}
{\frak z}&:=\left\{\left.
\begin{bmatrix}
c & 0\\
0 & c
\end{bmatrix}~\right|~c\in\R\right\},\quad\n:=\left\{\left.
\begin{bmatrix}
0 & x\\
0 & 0
\end{bmatrix}~\right|~x\in\Hl\right\},\\ \nonumber
\la&:=\left\{\left.
\begin{bmatrix}
t & 0\\
0 & -t
\end{bmatrix}\right|~t\in\R\right\},\quad
\mk:=\{X\in M_2(\Hl)\mid {}^t\bar{X}+X=0_2\},
\end{align}
where ${\frak z},~\n,~\la$ and $\mk$ are the Lie algebras of $Z^+,~N,~A$ and $K$ respectively.

We next consider the root space decomposition of $\g$ with respect to $\la$. Let $H:=
\begin{bmatrix}
1 & 0 \\
0 & -1
\end{bmatrix}$ and $\alpha$ be the linear form of $\la$ such that $\alpha(H)=1$. 
Then $\{\pm2\alpha\}$ is the set of roots for $(\g,\la)$. 
For $z\in\Hl$ we put
\[
E_{2\alpha}^{(z)}:=
\begin{bmatrix}
0 & z\\
0 & 0
\end{bmatrix},\quad E_{-2\alpha}^{(z)}:=
\begin{bmatrix}
0 & 0\\
z & 0
\end{bmatrix}.
\]
The set $\{E_{2\alpha}^{(1)},~E_{2\alpha}^{(i)},~E_{2\alpha}^{(j)},~E_{2\alpha}^{(k)}\}$~
(respectively~$\{E_{-2\alpha}^{(1)},~E_{-2\alpha}^{(i)},~E_{-2\alpha}^{(j)},~E_{-2\alpha}^{(k)}\}$)  forms a basis of $\n$~(respectively~a basis of $\bar{\n}:=\left\{\left.
\begin{bmatrix}
0 & 0\\
x & 0
\end{bmatrix}\right|~x\in\Hl\right\}$). Let ${\frak z}_{\la}(\mk):=\{X\in\mk\mid [X,A]=0~\forall A\in\la\}$, which coincides with
\[
\left\{\left.
\begin{bmatrix}
a & 0\\
0 & d
\end{bmatrix}~\right|~a,~d\in\Hl^-\right\}.
\]
Then ${\frak z}\oplus{\frak z}_{\la}(\mk)\oplus\la$ is the eigen-space with the eigenvalue zero. We then see from the root space decomposition of $\g$ with respect to $\la$ that $\g$ decomposes into 
\[
\g=({\frak z}\oplus{\frak z}_{\la}(\mk)\oplus\la)\oplus\n\oplus\bar{\n}.
\]

We also introduce the Lie group $\SL_2(\Hl)$ consisting of elements in $\GL_2(\Hl)$ with their reduced norms $1$. The Lie algebra $\g_0={\frak sl}_2(\Hl)$ of $\SL_2(\Hl)$ is the Lie algebra consisting of elements in $M_2(\Hl)$ with their reduced traces zero. 
For this we note that
\[
\GL_2(\Hl)/Z^+\simeq \SL_2(\Hl),\quad\g/{\frak z}\simeq\g_0.
\]
We introduce the differential operator $\Omega$ defined by the infinitesimal action of 
\begin{equation}\label{Casimir-defn}
\Omega := \frac{1}{16}H^2-\frac{1}{2}H+\frac{1}{4}\sum_{z\in\{1,i,j,k\}}{E_{2\alpha}^{(z)}}^2.
\end{equation} 
This differential operator $\Omega$ coincides with the infinitesimal action of the Casimir element of $\g_0$ (see \cite[p.293]{Kn}) on the space of right $K$-invariant smooth functions of $G/Z^+$. To check this we note $[E_{2\alpha}^{(z)},E_{-2\alpha}^{(\bar{z})}]=H$ for $z\in\Hl^1$ and Iwasawa decompositions $E_{-2\alpha}^{(z)}=E_{2\alpha}^{(z)}+
\begin{pmatrix}
0 & -z\\
z & 0
\end{pmatrix}$ for $z\in\Hl^-$. 
In what follows, we call $\Omega$ the Casimir operator.
\subsection{Automorphic forms}\label{Autom-form}
For $\lambda\in\C$ and a discrete subgroup $\Gamma\subset \SL_2(\R)$ we denote by $S(\Gamma,\lambda)$ the space of Maass cusp forms of weight $0$ on the complex upper half plane ${\frak h}$ whose eigenvalue with respect to the hyperbolic Laplacian is $-\lambda$. 

For a discrete subgroup $\Gamma\subset \GL_2(\Hl)$ and $r\in\C$ we denote by $\cM(\Gamma,r)$ the space of smooth functions $F$ on $\GL_2(\Hl)$ satisfying the following conditions: 
\begin{enumerate}
\item $\Omega\cdot F=-(\displaystyle\frac{r^2}{4}+1)F$, where $\Omega$ is the Casimir operator defined in \eqref{Casimir-defn},
\item for any $(z,\gamma,g,k)\in Z^+\times\Gamma\times G\times K$, we have $F(z\gamma gk)=F(g)$,
\item $F$ is of moderate growth.
\end{enumerate}
Let $K_{\alpha}$, with $\alpha\in\C$, denote the modified Bessel function (see \cite[Section 4.12]{A-A-R}), which satisfies the differential equation
\[
y^2\frac{d^2 K_{\alpha}}{dy^2}+y\frac{dK_{\alpha}}{dy}-(y^2+{\alpha}^2)K_{\alpha}=0.
\]
\begin{prop}\label{four-exp-prop}
Let $\Gamma$ be an arithmetic subgroup of $\GL_2(\Hl)$, and let $L_{\Gamma}:=\{x\in\Hl\mid n(x)\in N\cap\Gamma\}$ and $\hat{L_{\Gamma}}$ be the dual lattice of $L_{\Gamma}$ with respect to $\tr$. Then $F\in\cM(\Gamma,r)$ admits a Fourier expansion
\[
F(n(x)a_y)=u(y)+\sum_{\beta\in \hat{L_{\Gamma}}\setminus\{0\}}C(\beta)y^2K_{\sqrt{-1}r}(4\pi|\beta|y)e^{2\pi\sqrt{-1}\tr(\beta x)},
\]
with a smooth function $u$ on $\R_{>0}$.
\end{prop}
\begin{proof}
A Maass form $F\in \cM(\Gamma;r)$ is left-invariant with respect to $\{n(\beta)\mid\beta\in L_{\Gamma}\}$. This implies that $F(n(x+\alpha)g)=F(n(x)g)$ holds for $\alpha\in L_{\Gamma}$ and $g\in G$. Therefore $F$ has a expansion
\[
F(n(x)a_y)=\sum_{\beta\in\hat{L_{\Gamma}}}W_{\beta}(y)\exp2\pi\sqrt{-1}\tr(\beta x)
\]
with a smooth function $W_{\beta}$ on $\R_+^{\times}$. For $\xi\in\Hl\setminus\{0\}$ we put $\hat{W}_{\xi}(y):=y^{-\frac{3}{2}}W_{\xi}(y)$. From the condition $\Omega\cdot F=-(\displaystyle\frac{r^2}{4}+1)F$ we deduce that $\hat{W}_{\beta}$ satisfies the differential equation
\[
\left(\frac{d^2}{dY^2}+\left(-\frac{1}{4}+\frac{\frac{1}{4}+r^2}{Y^2}\right)\right)\hat{W}_{\beta}\left(\frac{Y}{8\pi|\beta|}\right)=0
\]
for $\beta\in\hat{L_{\Gamma}}\setminus\{0\}$, where $Y:=8\pi|\beta|y$. This is precisely the differential equation for the Whittaker function (see \cite[Section 4.3]{A-A-R}). With the Whittaker function $W_{0,\sqrt{-1}r}$ parametrized by $(0,\sqrt{-1}r)$ we thereby see that
\[
F(n(x)a_y)=u(y)+\sum_{\beta\in\hat{L_{\Gamma}}\setminus\{0\}}C^{\prime}(\beta)y^{\frac 32}W_{0,\sqrt{-1}r}(8\pi|\beta|y)\exp2\pi\sqrt{-1}\tr(\beta x),
\]
with constants $C^{\prime}(\beta)$ depending only on $\beta$.
. We now note the relation
\[
W_{0,\sqrt{-1}r}(2y)=\sqrt{\frac{2y}{\pi}}K_{\sqrt{-1}r}(y)
\]
(see \cite[Section 13, 13.18 (iii)]{O-L-B-C}). 
This means that $F$ has the Fourier expansion as in the statement of the proposition.
\end{proof}

We shall consider the automorphic forms above with specified discrete subgroups of $\SL_2(\R)$ and $\GL_2(\Hl)$. As a discrete subgroup of $\SL_2(\R)$ we take the congruence subgroup $\Gamma_0(2)$ of level $2$. For a choice of a discrete subgroup of $\GL_2(\Hl)$ we recall that $B$ denotes the definite quaternion algebra over $\Q$ with discriminant $d_B=2$. 
This has a unique maximal order $\mO$ given by
\[
\mO=\Z +\Z i+\Z j+\Z\frac{1+i+j+ij}{2},
\]
which is called the Hurwitz order. 
 As a discrete subgroup of $GL_2(\Hl)$ we mainly take $GL_2(\mO)$. 
\begin{prop}\label{group-gen-prop}
The group $\GL_2(\mO)$ is generated by
\[
\left\{\left.
\begin{bmatrix}
0 & 1\\
-1 & 0
\end{bmatrix},\quad
\begin{bmatrix}
u & 0\\
0 & 1
\end{bmatrix},\quad
\begin{bmatrix}
1 & v\\
0 & 1
\end{bmatrix}~\right|~u\in\mO^{\times},~v\in\mO\right\}.
\]
\end{prop}
\begin{proof}
Any element of the form $
\begin{bmatrix}
\alpha & \beta\\
0 & \delta
\end{bmatrix}\in \GL_2(\mO)$ can be expressed as
\[
\begin{bmatrix}
\alpha & 0\\
0 & 1
\end{bmatrix}
\begin{bmatrix}
1 & 0\\
0 & \delta
\end{bmatrix}
\begin{bmatrix}
1 & \alpha^{-1}\beta\\
0 & 1
\end{bmatrix}.
\]
We see that $\alpha,~\delta\in\mO^{\times}$ and thus $\alpha^{-1}\beta\in\mO$. 
We note that 
\[
\begin{bmatrix}
0 & 1\\
-1 & 0
\end{bmatrix}
\begin{bmatrix}
u & 0\\
0 & 1
\end{bmatrix}
\begin{bmatrix}
0 & -1\\
1 & 0
\end{bmatrix}=
\begin{bmatrix}
1 & 0\\
0 & u
\end{bmatrix},\quad
\begin{bmatrix}
0 & 1\\
-1 & 0
\end{bmatrix}^3=
\begin{bmatrix}
0 & -1\\
1 & 0
\end{bmatrix}.
\]
These imply the assertion for 
$\begin{bmatrix}
\alpha & \beta\\
0 & \delta
\end{bmatrix}\in \GL_2(\mO)$. Next, we have the following claim.

{\it For $a,~b\in\mO$ with $b\not=0$ there exists $c,~d\in\mO$ such that
$
a=cb+d,\quad\nu(d)<\nu(b)
$.}
This follows from \cite[Chapter I,~Section 1,~Corollary 1.8]{Kr}. This reduces the general case of $\begin{bmatrix}
\alpha & \beta\\
\gamma & \delta
\end{bmatrix}\in \GL_2(\mO)$, $\gamma\not=0$, to the previous case. This completes the proof of the proposition.
\end{proof}

Let  $\cG(\A)=\GL_2(B_{\A})$, where $B_{\A}$ denotes the adelization of $B$, and let $U$ be the compact subgroup of $\cG(\A)$ given by $\prod_{p<\infty}\GL_2(\mO_p)$, where $\mO_p$ denotes the $p$-adic completion of $\mO$ at a finite prime $p$. 
Then the class number of $\cG$ with respect to $U$ is defined as the number of cosets in $U\cG(\R)\backslash\cG(\A)/\cG(Q)$. 

We next put ${\cal P}$ to be a standard $\Q$-parabolic subgroup of $\cG$ whose group of $\Q$-rational points is ${\cal P}(\Q)=\left\{
\begin{bmatrix}
\alpha & \beta\\
0 & \delta
\end{bmatrix}\in\cG(\Q)\right\}$. We now recall that, for an arithmetic subgroup $\Gamma\subset{\cal G}(\Q)$, the cosets $\Gamma\backslash{\cal G}(\Q)/{\cal P}(\Q)$ are called the set of $\Gamma$-cusps. 
\begin{lem}\label{cusp-lem}
\begin{enumerate}
\item The class number of $\cG$ with respect to $U$ is one, namely we have ${\cal G}(\A)={\cal G}(\Q){\cal G}(\R)U$.
\item The number of cusps with respect to $\Gamma:=\GL_2(\mO)$ is one.
\end{enumerate}
\end{lem}
\begin{proof}
According to \cite[Theorem 8.11]{P-R} the class number of a reductive group over a number field $F$ is not greater than that of a parabolic $F$-subgroup of it. Furthermore the class number of a parabolic $F$-subgroup is not greater than that of its Levi subgroup~(see \cite[Proposition 5.4]{P-R}). Hence, the class number of ${\cal G}$ with respect to $U$ turns out to be not greater than that of the Levi subgroup ${\cal L}$ defined by the $Q$-rational points $B^{\times}\times B^{\times}$. Since the class number of $B^{\times}$ with respect to $\prod_{p<\infty}\mO_p^{\times}$ is one, the class number of ${\cal G}$ is also one, which means $\cG(\A)=\cG(\Q){\cal G}(\R)U=U\cG(\R)\cG(\Q)$. This completes the proof of 1.

The first assertion implies that there is a bijection
\[
\Gamma\backslash {\cG}(\Q)/{\cal P}(\Q)\simeq {\cG}(\R)U\backslash {\cG}(\A)/{\cal P}(\Q).
\]
Furthermore note an Iwasawa decomposition
\[
{\cal G}(\Q_v)=
\begin{cases}
P(\Q_p)\cdot \GL_2(\mO_p)&(v=p<\infty)\\
P(\R)\cdot K &(v=\infty)
\end{cases}
\]
at every place $v\le\infty$. We can then reduce the counting of the number of $\Gamma$-cusps to that of the class number of the Levi subgroup ${\cal L}$. This implies that the number of cusps with respect to $\Gamma$ is one and completes the proof of the lemma. 
\end{proof}

We define $\Gamma_T$ as a subgroup $\GL_2(\mO)$ generated by
\begin{equation}\label{Gamma-T-defn}
\begin{bmatrix}
0 & -1\\
1 & 0
\end{bmatrix},\quad
\begin{bmatrix}
1 & \beta\\
0 & 1
\end{bmatrix}\quad(\beta\in\mO).
\end{equation}
In what follows, we will deal mainly with $S(\Gamma_0(2);-(\frac{1}{4}+(\frac{r}{2})^2))$, $\cM(\Gamma_T;r)$ and $\cM(\GL_2(\mO),r)$. 
For this we should note that $r$ can be assumed to be in $\R$ since the Selberg conjecture for $\Gamma_0(2)$ is verified~(cf.~\cite[Corollary 11.5]{Iwn}). By Proposition \ref{four-exp-prop}, the Fourier expansion of $F\in\cM(\GL_2(\mO),r)$ is then written as
\begin{align*}
F(n(x)a_y)&=u(y)+\sum_{\beta\in \frac{1}{2}S\setminus\{0\}}C(\beta)y^2K_{\sqrt{-1}r}(4\pi\nu(\beta)y)e^{2\pi\sqrt{-1}\tr(\beta x)}\\
&=u(y)+\sum_{\beta\in S\setminus\{0\}}A(\beta)y^2K_{\sqrt{-1}r}(2\pi|\beta|y)e^{2\pi\sqrt{-1}\real(\beta x)}
\end{align*}
with a smooth function $u$ on $\R_{>0}$.
Here 
\begin{equation}\label{S-defn}
S:=\Z\cdot(1-ij)+\Z\cdot(-i-ij)+\Z\cdot(-j-ij)+\Z\cdot 2ij
\end{equation}
is the dual lattice of $\mO$ with respect to the bilinear form on $\Hl\times\Hl$ defined by $\real=\frac{1}{2}\tr$. 

Now we introduce $\varpi_2=1+i$, which is a uniformizer of $B{\otimes}_{\Q}\Q_2$. We can verify the following lemma by a direct computation.
\begin{lem}\label{S-lem}
We have $S=\varpi_2\mO$.
\end{lem}
\section{Some zeta integral of convolution type}\label{zeta-int-sec}
In this section, we study certain zeta integrals which will play a crucial role in the proof of automorphy in Section \ref{Construction-lifting}. 
\subsection*{Theta functions}
Consider the space of harmonic polynomials of degree $l$  on $\Hl$. These are homogeneous polynomials of degree $l$ and are annihilated by the Laplace operator in $4$ variables. We can act on this space by the cyclic group of order $8$ generated by $\frac{1+i}{\sqrt{2}} \in \Hl$. Let $\{P_{l,\nu}\}_{\nu}$ denote a basis for this space consisting of eigenvectors under the above action. Hence, 
\begin{equation}\label{P-eigen-reln}
P_{l,\nu}(\frac{1+i}{\sqrt{2}} x) = \epsilon_{l,\nu} P_{l,\nu}(x),
\end{equation}
for some eighth root of unity $\epsilon_{l,v}$. Define the following theta function
\begin{equation}\label{theta-defn}
\Theta_{l,\nu}(z) := \sum\limits_{\beta \in S} P_{l,\nu}(\beta) e^{2\pi\sqrt{-1}\frac{|\beta|^2}{2}z}=\sum_{m=0}^{\infty}b(2m)e^{2\pi\sqrt{-1}mz}
\end{equation}
on ${\frak h}$, where $b(m):=\sum_{\beta\in S,~|\beta|^2=m}P_{l,\nu}(\beta)$. Since, $S$ is invariant under $\beta \mapsto - \beta$ and $P_{l,\nu}(-x) = (-1)^l P_{l,\nu}(x)$, we see that $\Theta_{l,\nu}(z)$ is the zero function if $l$ is odd.
\begin{lem}\label{theta-property-lem}
Let $l$ be an even non-negative integer. Let $\Theta_{l,\nu}$ be as defined in \eqref{theta-defn}, with $P_{l,\nu}$ satisfying \eqref{P-eigen-reln}. Then $\Theta_{l,\nu}$ is a holomorphic modular form of weight $l+2$ with respect to $\Gamma_0(2)$, and is a cusp form if $l \geq 2$. Moreover, we have the following transformation formula
\begin{equation}\label{theta-transf-formula}
\Theta_{l,\nu}(\frac{-1}{2z}) = -\epsilon_{l,\nu}^{-1} 2^{\frac l2 + 1} z^{l+2} \Theta_{l,\nu}(z).
\end{equation}
\end{lem}
\begin{proof}
Set 
$$B = \begin{bmatrix}1\\&1\\&&1\\-1&-1&-1&2\end{bmatrix} \qquad \text{ and } \qquad A = {}^{t}B B.$$
By \eqref{S-defn}, we see that the map $x \mapsto B x$ is a bijection from $\Z^4$ to $S$. Here, we consider $x$ as a column vector. For $P_{l,\nu}$ in the statement of the lemma, set $P(x) := P_{l,\nu}(Bx), x \in \Z^4$. Then $P$ is a homogeneous polynomial in $4$ variables of degree $l$ annihilated by the operator
$$\Delta_A = \sum\limits_{i,j=1}^4 b_{i,j} \frac{\partial^2}{\partial \, x_i \partial \, x_j}, \qquad \text{ where } \qquad A^{-1} = (b_{i,j}).$$
One can then see that $\Theta_{l,\nu}(z) = \Theta(z; A,P)$, where 
$$\Theta(z;A,P) = \sum\limits_{m \in \Z^4} P(m) e^{2 \pi \sqrt{-1} \frac{{}^{t}m A m}2 z}$$
is as defined in \cite[Corollary 4.9.5]{Mi}. Since, all diagonal entries of $A$ and $2A^{-1}$ are even, part (3) of Corollary 4.9.5 of \cite{Mi} implies that $\Theta_{l,\nu}$ is a holomorphic modular form of weight $l+2$ with respect to $\Gamma_0(2)$, and is a cusp form if $l \geq 2$. Once again, by part (3) of Corollary 4.9.5 of \cite{Mi}, we have
$$\Theta(\frac{-1}{2z}; A,P) = -2^{l+1} z^{l+2} \Theta(z; A^\ast, P^\ast),$$
where $A^\ast = 2 A^{-1}$ and $P^\ast(x) = P(A^{-1}x) = P_{l,\nu}({}^{t}B^{-1}x)$. Note, that the map $x \mapsto {}^{t}B^{-1}x$ gives a bijection between $\Z^4$ and $\mO$. Hence, we see that 
$$\Theta(z; A^\ast, P^\ast) = \sum\limits_{\beta \in \mO} P_{l,\nu}(\beta) e^{2\pi\sqrt{-1}|\beta|^2 z}= \epsilon_{l,\nu}^{-1} 2^{-\frac l2} \sum_{m=0}^{\infty}b(2m)e^{2\pi\sqrt{-1}mz} = \epsilon_{l,\nu}^{-1} 2^{-\frac l2} \Theta_{l,\nu}(z).$$
Here, we have used Lemma \ref{S-lem} and \eqref{P-eigen-reln}. This completes the proof of the lemma.
\end{proof}

\subsection*{Eisenstein series with respect to $\Gamma_0(2)$}
We introduce an Eisenstein series
\begin{equation}\label{Eis-defn}
\tilde{E}_{\infty}(z,s):=(4\pi)^{\frac{l}{2}}\frac{\Gamma(s+\frac{1}{2}+l)}{\Gamma(s)}(\pi^{-s}\Gamma(s)\zeta(2s))\frac{1}{2}\sum_{\gamma\in\Gamma_{\infty}\backslash\Gamma_0(2)}(\frac{cz+d}{|cz+d|})^{l+2}(\frac{\IM(z)}{|cz+d|^2})^s
\end{equation}
on ${\frak h}$ with a complex parameter $s$, where $\Gamma_{\infty}:=\left\{\left.
\begin{bmatrix}
1 & m\\
0 & 1
\end{bmatrix}~\right|~m\in\Z\right\}$. The Eisenstein series satisfies the following functional equation.

\begin{lem}\label{Eis-fnal-eqn}
Let $\tilde{E}_0(z,s):=(\frac{z}{|z|})^{l+2}\tilde{E}_{\infty}(\frac{-1}{2z},s)$. Then the functional equation
\[
\tilde{E}_{\infty}(z,1-s)=\frac{2^{2s-2}}{1-2^{2s-2}}\tilde{E}_{\infty}(z,s)+\frac{2^{s-1}(1-2^{2s-1})}{1-2^{2s-2}}\tilde{E}_0(z,s)
\]
holds.
\end{lem}
\begin{proof}
This is settled by the argument in \cite[Lemma 7]{D-Im} with the help of the formula for  the scattering matrices in \cite[Section 11.2]{Iwn}.
\end{proof}

Let $f\in S(\Gamma_0(2);-(\frac{1}{4}+(\frac{r}{2})^2))$ be an eigenfunction with respect to the Atkin-Lehner involution, i.e. $f(\frac{-1}{2z})=\epsilon f(z)$ with some $\epsilon\in\{\pm 1\}$. Suppose the form $f$ has the Fourier expansion
\[
f(x+\sqrt{-1}y)=\sum_{n\in\Z\setminus\{0\}}c(n)W_{0,\frac{\sqrt{-1}r}{2}}(4\pi|n|y)e^{2\pi\sqrt{-1}nx},
\]
where $W_{0,\frac{\sqrt{-1}r}{2}}$ denotes the Whittaker function with the parameter $(0,\frac{\sqrt{-1}r}{2})$. Let $l$ be an even non-negative integer. Let $\Theta_{l,\nu}$ be as defined in \eqref{theta-defn}, with $P_{l,\nu}$ satisfying \eqref{P-eigen-reln}. Let $\tilde{E}_{\infty}(z,s)$ be the Eisenstein series defined in \eqref{Eis-defn}. Let the zeta integral $I(s)$ be defined by
\begin{equation}\label{zeta-int-defn}
I(s):=\int_{\Gamma_0(2)\backslash{\frak h}}f(z)\Theta_{l,\nu}(z)\tilde{E}_{\infty}(z,s)y^{\frac{l+2}{2}}\frac{dxdy}{y^2}.
\end{equation}
By Lemma \ref{theta-property-lem}, the above integral is well-defined. Let us now state the theorem of this section.
\begin{thm}\label{zeta-int-fnal-eqn-thm}
The zeta integral $I(s)$ is entire and is bounded on vertical strips. When $\epsilon\epsilon_{l,\nu}=1$ we have
\[
(2^s-1)I(s)=(2^{1-s}-1)I(1-s).
\]
When $\epsilon\epsilon_{l,\nu}=-1$ we have
\[
(2^s+1)I(s)=(2^{1-s}+1)I(1-s).
\]
\end{thm}
\begin{proof}
The entireness and boundedness on vertical strips of $I(s)$ is verified by the same argument as \cite[Section 3.2]{P}. 
We put
$$I_0(s):=\int_{\Gamma_0(2)\backslash{\frak h}}f(z)\Theta_{l,\nu}(z)\tilde{E}_0(z,s)y^{\frac{l+2}{2}}\frac{dxdy}{y^2}.$$
Since, $\Gamma_0(2)$ is stable under conjugation by $\mat{}{1/\sqrt{2}}{-\sqrt{2}}{} \in \SL_2(\R)$, we can make a change of variable $z \mapsto -1/(2z)$. Now, using the assumption $f(-\frac{1}{2z})=\epsilon f(z)$,  \eqref{theta-transf-formula} and the definition of $\tilde{E}_0(z,s)$, we have
\begin{align*}
I_0(s)&=\int_{\Gamma_0(2)\backslash{\frak h}}f(-\frac{1}{2z})\Theta_{l,\nu}(-\frac{1}{2z})\tilde{E}_0(-\frac{1}{2z},s){\rm Im}\big(\frac{-1}{2z}\big)^{\frac{l+2}{2}}\frac{dxdy}{y^2} \\
&= \int_{\Gamma_0(2)\backslash{\frak h}} \Big[\epsilon f(z) \Big] \Big[-\epsilon_{l,\nu}^{-1} 2^{\frac l2 + 1} z^{l+2} \Theta_{l,\nu}(z)\Big] \Big[\big(\frac{-|z|}{z}\big)^{l+2} \tilde{E}_{\infty}(z,s)\Big] \Big[\frac y{2|z|^2}\Big]^{\frac{l+2}{2}}\frac{dxdy}{y^2} \\
&= -\epsilon_{l,v}^{-1}\epsilon I(s).
\end{align*}
From Lemma \ref{Eis-fnal-eqn}, we deduce
\[
I(1-s)=\frac{2^{2s-2}}{1-2^{2s-2}}I(s)+\frac{2^{s-1}(1-2^{2s-1})}{1-2^{2s-2}}I_0(s).
\]
Observe that 
\[
\frac{2^{2s-2}}{1-2^{2s-2}}-\epsilon_{l,\nu}^{-1}\epsilon\frac{2^{s-1}(1-2^{2s-1})}{1-2^{2s-2}}=
\begin{cases}
\displaystyle\frac{2^s-1}{2^{1-s}-1}& \text{if } \epsilon_{l,\nu}\epsilon=1,\\
\displaystyle\frac{2^s+1}{2^{1-s}+1}& \text{ if } \epsilon_{l,\nu}\epsilon=-1.
\end{cases}
\]
We have therefore proved the theorem.
\end{proof}

\section{Construction of the lifting}\label{Construction-lifting}
We now construct a lifting map from $S(\Gamma_0(2);-(\frac{1}{4}+\frac{r^2}{4}))$ to $\cM(\GL_2(\mO);r)$, which is an analogue of Pitale \cite{P}. The fundamental tool of our study is the converse theorem by Maass \cite{Ma}. 
\begin{thm}[Maass] \label{maass-thm}
Let $\{A(\beta)\}_{\beta\in S\setminus\{0\}}$ be a sequence of complex numbers such that
\[
A(\beta)=O(|\beta|^{\kappa})\quad(\exists\kappa>0)
\]
and put 
\[
F(n(x)a_y):=\sum_{\beta\in S\setminus\{0\}}A(\beta)y^2K_{\sqrt{-1}r}(2\pi|\beta|y)e^{2\pi\sqrt{-1}\real(\beta x)}.
\]
For a harmonic polynomial $P$ on $\Hl$ of degree $l$ we introduce
\[
\xi(s,P):=\pi^{-2s}\Gamma(s+\frac{\sqrt{-1}r}{2})\Gamma(s-\frac{\sqrt{-1}r}{2})\sum_{\beta\in S\setminus\{0\}}A(\beta)\frac{P(\beta)}{|\beta|^{2s}},
\]
which converges for $\real(s)>\frac{l+4+\kappa}{2}$. Let $\{P_{l,\nu}\}_\nu$ be a basis of harmonic polynomials on $\Hl$ of degree $l$.

Then $F\in\cM(\Gamma_T;r)$ is equivalent to the condition that, for any $l, \nu$, the $\xi(s,P_{l,\nu})$ satisfies the following three conditions.
\begin{enumerate}
\item it has analytic continuation to the whole complex plane.
\item it is bounded on any vertical strip of the complex plane.
\item the functional equation
\[
\xi(2+l-s,P_{l,\nu})=(-1)^l\xi(s,\hat{P}_{l,\nu})
\]
holds, where $\hat{P}(x):=P(\bar{x})$ for $x\in\Hl$.
\end{enumerate}
\end{thm}
Recall that $\Gamma_T$ is defined in \eqref{Gamma-T-defn}. For this theorem we remark that the infinitesimal action of the Casimir operator $\Omega$ on the space of smooth right $K$-invariant functions on $G/Z$ can be identified with a constant multiple of the hyperbolic Laplacian on $\left\{\left.
\begin{bmatrix}
y & x\\
0 & 1
\end{bmatrix}~
\right|~x\in\Hl,y\in\R_+^{\times}\right\}$~(for the hyperbolic Laplacian see \cite[(3)]{Ma}). We can therefore follow the argument in \cite{Ma} to see that this theorem is useful also for our situation. 

We wish to define $\{A(\beta)\}_{\beta\in S\setminus\{0\}}$ from Fourier coefficients $c(n)$ of $f\in S(\Gamma_0(2);-(\frac{1}{4}+\frac{r^2}{4}))$. 
Let $\varpi_2 = 1+i$, as before. An easy computation shows that $\varpi_2 \mO = \mO \varpi_2$. This allows us to write any $\beta \in \mO$ uniquely as $\beta = \varpi_2^u d \beta'$, where $u \geq 0$, an odd integer $d$ and $\beta' \in \mO$ is neither of the form $\varpi_2\beta'_0$ with some non-zero $\beta'_0\in\mO$ nor a multiple of an element of $\mO$ by an odd integer. 
Hence, we can define $\varpi_2^m|\beta$ by $m\le u$ with $u$ as above.  Recall that, by Lemma \ref{S-lem}, we have $S = \varpi_2 \mO$. It thus makes sense to define the set $S^{\rm{prim}}$ of primitive elements in $S$ by
\begin{equation}\label{S-prim-defn}
S^{\rm{prim}}:=\{\beta\in S\setminus\{0\}\mid \varpi_2\mid\beta,~\varpi_2^2\nmid\beta,~d\nmid\beta~\text{for all odd integer $d$}\},
\end{equation}
where ``$d\nmid\beta$'' for $d\in\Z$ means that $\beta$ is not a multiple of an element in $S$ by $d$.
\begin{prop}\label{main-prop}
Let $\beta\in S\setminus\{0\}$ be expressed as
\[
\beta={\varpi_2}^ud\beta_0,
\]
where $u$ is a non-negative integer, $d$ an odd integer and $\beta_0 \in S^{\rm prim}$.
Given $f\in S(\Gamma_0(2);-(\frac{1}{4}+\frac{r^2}{4}))$ with Fourier coefficients $c(n)$ and eigenvalue $\epsilon\in\{\pm 1\}$ of the Atkin-Lehner involution, we set
\begin{equation}\label{Abeta-defn}
A(\beta):=|\beta|\sum_{t=0}^u\sum_{n|d}(-\epsilon)^{t}c(-\frac{|\beta|^2}{2^{t+1}n^2}).
\end{equation}
Let $\{P_{l,\nu}\}_\nu$ be a basis of harmonic polynomials on $\Hl$ of degree $l$ satisfying \eqref{P-eigen-reln}. Then we have
\begin{equation}\label{xi=I}
\xi(s+\frac{l}{2}+\frac{1}{2},P_{l,\nu})=
\begin{cases}
2^{1-\frac{l}{2}}\pi^{-(l+1)}(2^s-\epsilon\epsilon_{l,\nu})I(s)& \text{ if } \epsilon_{l,\nu} \in\{\pm 1\},\\
0& \text{ if } \epsilon_{l,\nu}\not\in\{\pm 1\})
\end{cases}
\end{equation}
and the $\xi(s,P_{l,\nu})$ satisfies the three analytic conditions in Theorem \ref{maass-thm}.
\end{prop}
\begin{proof}
Note that, if $l$ is odd, then $\xi(s, P_{l,\nu}) \equiv 0$, since $A(\beta)$ is invariant under $\beta \mapsto -\beta$ and $P_{l,\nu}$ is homogeneous of degree $l$. For this we remark that, from Section \ref{zeta-int-sec}, $I(s) \equiv 0$ also holds when $l$ is odd. From now on we will assume that $l$ is even. Now suppose that the condition $\epsilon_{l,\nu}\not\in\{\pm 1\}$ is satisfied, which is equivalent to $\epsilon_{l,\nu}^2\not=1$. We see that $P_{l,\nu}(i\beta)=\epsilon_{l,\nu}^2P_{l,\nu}(\beta)\not=P_{l,\nu}(\beta)$ for $\beta\in\Hl$. In addition we note that $A(i\beta)|i\beta|^{-2s}=A(\beta)|\beta|^{-2s}$ for $\beta\in S$.  In the definition of $\xi(s,P_{l,\nu})$ we can replace $S$ by $iS$. Hence, we obtain $\xi(s,P_{l,\nu})\equiv 0$, if $\epsilon_{l,\nu}\not\in\{\pm 1\}$.

We now assume that $\epsilon_{l,\nu}\in\{\pm 1\}$.  By a formal calculation similar to \cite[Proposition 3.5]{P} we get
\[
I(s)=\pi^{-2s}2^{-2s-1}\Gamma(s+\frac{l}{2}+\frac{1}{2}+\frac{\sqrt{-1}r}{2})\Gamma(s+\frac{l}{2}+\frac{1}{2}-\frac{\sqrt{-1}r}{2})\zeta(2s)\sum_{m=1}^{\infty}\frac{c(-m)b(2m)}{m^{s+\frac{l}{2}}}.
\]
Put $\Gamma_{l,r}(s):=\pi^{-(2s+l+1)}\Gamma(s+\frac{l}{2}+\frac{1}{2}+\frac{\sqrt{-1}r}{2})\Gamma(s+\frac{l}{2}+\frac{1}{2}-\frac{\sqrt{-1}r}{2})$.
We have
\begin{align*}
\xi(s+\frac{l}{2}+\frac{1}{2},P_{l,\nu})&=\Gamma_{l,r}(s)\sum_{\beta\in S\setminus\{0\}} \frac{A(\beta)P_{l,\nu}(\beta)}{|\beta|^{2(s+\frac{l}{2}+\frac{1}{2})}}\\
&=\Gamma_{l,r}(s)\sum_{\beta\in S\setminus\{0\}}\frac{\sum_{t=0}^{u}\sum_{n|d}c(\frac{-|\beta|^2}{2^{t+1}n^2})(-\epsilon)^tP_{l,\nu}(\beta)}{(|\beta|^2)^{s+\frac{l}{2}}}\\
&=\Gamma_{l,r}(s)\sum_{\beta\in S\setminus\{0\}}\sum_{t=0}^{u}\sum_{n|d}\frac{c(-\frac{|\beta|^2}{2^{t+1}n^2})(-\epsilon)^tP_{l,\nu}(\frac{\beta}{2^{t/2}n})}{(2^tn^2)^s(\frac{|\beta|^2}{2^tn^2})^{s+\frac{l}{2}}}\\
&=\Gamma_{l,r}(s)\sum_{u=0}^\infty \sum_{\substack{d \geq 1 \\ d:\rm{odd}}} \sum_{\beta\in S^{\rm prim}}\sum_{t=0}^u\sum_{n|d}\frac{c(-\frac{1}{2}|\varpi_2^{u-t}\frac{d}{n}\beta_0|^2)(-\epsilon\epsilon_{l,\nu})^tP_{l,\nu}(\varpi_2^{u-t}\frac{d}{n}\beta_0)}{(2^tn^2)^s|\varpi_2^{u-t}\frac{d}{n}\beta_0|^{2(s+\frac{l}{2})}}\\
&=\Gamma_{l,r}(s)\sum_{u=0}^{\infty}\sum_{\substack{d\ge 1 \\ d:\rm{odd}}}\left(\sum_{t=0}^{u}\sum_{n|d}\sum_{\beta\in S^{\rm{prim}}}\frac{c(-\frac{1}{2}|\varpi_2^{t}n\beta|^2)P_{l,\nu}(\varpi_2^{t}n\beta)}{(-2^s\epsilon_{l,\nu}\epsilon)^{u-t}(\frac{d}{n})^{2s}|\varpi_2^{t}n\beta|^{2(s+\frac{l}{2})}}\right)\\
&=\Gamma_{l,r}(s)\sum_{u=0}^{\infty}\frac{1}{(-2^s\epsilon_{l,\nu}\epsilon)^u}\sum_{\substack{d\ge 1 \\ d:\rm{odd}}}\frac{1}{d^{2s}}\sum_{m=1}^{\infty}\frac{c(-m)\sum\limits_{\beta \in S, |\beta|^2=2m}P_{l,\nu}(\beta)}{(2m)^{s+\frac{l}{2}}}\\
&=2^{1-\frac{l}{2}}\frac{2^{2s}-1}{2^s+\epsilon_{l,\nu}\epsilon}2^{-2s-1}\Gamma_{l,r}(s)\zeta(2s)\sum_{m=1}^{\infty}\frac{c(-m)b(2m)}{m^{s+\frac{l}{2}}}\\
&=2^{1-\frac{l}{2}}\pi^{-(l+1)}(2^s-\epsilon_{l,\nu}\epsilon)I(s).
\end{align*}
We thus have $\xi(s,\hat{P}_{l,\nu})=\xi(s,P_{l,\nu})$ since $S=\bar{S}$. The formula just proved and Theorem \ref{zeta-int-fnal-eqn-thm} imply that $\xi(s,P_{l,\nu})$ satisfies the desired three analytic properties in Theorem \ref{maass-thm}.
\end{proof}
\begin{thm}\label{lift-thm}
Let $f\in S(\Gamma_0(2);-(\frac{1}{4}+\frac{r^2}{4}))$ with Fourier coefficients $c(n)$ and with eigenvalue $\epsilon$ of the Atkin Lehner involution. Define 
$$F_f(n(x)a_y):=\sum_{\beta\in S\setminus\{0\}}A(\beta)y^2K_{\sqrt{-1}r}(2\pi|\beta|y)e^{2\pi\sqrt{-1}\real(\beta x)}$$
with $\{A(\beta)\}_{\beta\in S\setminus\{0\}}$ defined by \eqref{Abeta-defn}. Then we have $F_f\in {\cal M}(\GL_2(\mO);r)$ and $F_f$ is a cusp form. 
Furthermore, $F_f\not\equiv 0$.
\end{thm}
\begin{proof}
We can verify the left invariance of $F_f$ with respect to $\{
\begin{bmatrix}
u & 0\\
0 & 1
\end{bmatrix}\mid u\in\mO^{\times}\}$ in a straightforward way. 
Proposition \ref{group-gen-prop}, Theorem \ref{maass-thm} and Proposition \ref{main-prop} thus imply $F_f\in{\cal M}(GL_2(\mO);r)$. Since $GL_2(\mO)$ has only one cusp (see Lemma \ref{cusp-lem}), the Fourier expansion of $F_f$ means that $F_f$ is cuspidal. 
To show the non-vanishing we need the following lemma:

\begin{lem}\label{Maass-form-non-van-coeff}
Let $f\in S(\Gamma_0(2);-(\frac{1}{4}+\frac{r^2}{4}))$ with Fourier coefficients $c(n)$ and with eigenvalue $\epsilon$ of the Atkin Lehner involution. Then, there exist $N > 0, N \in \Z$, such that $c(-N) \neq 0$.
\end{lem}
\begin{proof}
Assume that $c(n) = 0$ for all $n < 0$. Set $f_1(z) = (f(z)+f(-\bar{z}))/2$ and $f_2(z) = (f(z)-f(-\bar{z}))/2$. Then, $f_1, f_2$ are elements of $S(\Gamma_0(2);-(\frac{1}{4}+\frac{r^2}{4}))$ with the same  eigenvalue $\epsilon$ of the Atkin Lehner involution as $f$. In addition, $f_1$ is an even Maass form and $f_2$ is an odd Maass form, with the property that they have the exact same Fourier coefficients corresponding to positive indices. This implies that the $L$-functions for $f_1$ and $f_2$ satisfy $L(s, f_1) = L(s, f_2)$. On the other hand, $L(s,f_1)$ and $L(s,f_2)$ satisfy functional equations with the gamma factors shifted by $1$. Here, we use that both $f_1$ and $f_2$ have the same Atkin Lehner eigenvalue. If $L(s, f_1) \neq 0$, we obtain an identity of gamma factors, which can be checked to be impossible. This gives us that $f$ has to be zero, a contradiction.
\end{proof}

Let $N_0$ be the smallest positive integer such that $c(-N_0)  \neq 0$. Let $\beta_0 \in \mO$ be such that $|\beta_0|^2 = N_0$. Choose, $\beta = \varpi_2 \beta_0$. Then, by the choice of $N_0$ and definition of $A(\beta)$, we see that $A(\beta) = \sqrt{2N_0} c(-N_0) \neq 0$, as required.
\end{proof}
\begin{rem}\label{Existence-Maass-Cuspforms}
Weyl's law for congruence subgroups of $SL_2(\Z)$ by Selberg~(cf.~\cite[Section 11.1]{Iwn}) implies that there exist Maass cusp forms for $\Gamma_0(2)$.  
This and the theorem above imply the existence of non-zero lifts $F_f$.
\end{rem}
\section{Actions of Hecke operators on the lifting}\label{Hecke-action}
\subsection{Adelization of automorphic forms}\label{adel-section}
To study the actions of Hecke operators on our cusp forms constructed by the lifting we need both adelic and non-adelic treatments of automorphic forms. 

For a complex number $r\in\C$ we introduce another space $M(\cG(\A),r)$ of automorphic forms for $\cG$.
\begin{defn}
Let $M(\cG(\A),r)$ be the space of smooth functions $\Phi$ on $\cG(\A)$ satisfying the following conditions:
\begin{enumerate}
\item $\Phi(z\gamma gu_fu_{\infty})=\Phi(g)$ for any $(z,\gamma,g,u_f,u_{\infty})\in Z_{\A}\times\cG(\Q)\times\times\cG(\A)\times U\times K$, where $Z_{\A}$ denotes the center of $\cG(\A)$,
\item $\Omega\cdot\Phi(g_{\infty})=-(\displaystyle\frac{r^2}{4}+1)\Phi(g_{\infty})$ for any $g_{\infty}\in\cG(\R)=\GL_2(\Hl)$,
\item $\Phi$ is of moderate growth.
\end{enumerate}
\end{defn}
According to part 1) of Lemma \ref{cusp-lem}, the class number of $\cG$ with respect to $U$ is one, which means that $\cG(\A)=\cG(\Q)\cG(\R)U$. We can thus view $F\in M(\GL_2(\mO),r)$ as a smooth function $\Phi_F$ on $\cG(\A)$ by
\[
\Phi_F(\gamma g_{\infty}u_f)=\Phi_F(g_{\infty})\quad\forall(\gamma,g_{\infty},u_f)\in\cG(\Q)\times\cG(\R)\times U.
\]
We therefore see the following:
\begin{lem}\label{adelization}
We have an isomorphism $M(\GL_2(\mO),r)\simeq M(\cG(\A),r)$.
\end{lem}
\subsection{Hecke operators}
For each place $p\le\infty$ let $G_p:=\GL_2(B_p)$ with $B_p=B\otimes_{\Q}{\Q_p}$. For a finite prime $p\not=2$, we have $\GL_2(B_p)\simeq \GL_4(\Q_p)$. Let $\mO_p$ be the $p$-adic completion of $\mO$ for $p<\infty$. For a finite prime  $p\not=2$, $\mO_p\simeq M_2(\Z_p)$ and $\GL_2(\mO_p)\simeq \GL_4(\Z_p)$. Set $K_p = \GL_2(\mO_p)$ for $p<\infty$.

We denote by $\Hp$ the Hecke algebra for $\GL_2(B_p)$ with respect to $\GL_2(\mO_p)$ for $p<\infty$. According to \cite[Section 8,~Theorem 6]{Sa}, $\Hp$ has the following generators:
\[
\begin{cases}
\{\varphi_1^{\pm 1},~\varphi_2\}& \text{ if } p=2,\\
\{\phi_1^{\pm 1},~\phi_2,~\phi_3,~\phi_4\}& \text{ if } p\not=2.
\end{cases}
\]
Here $\varphi_1,~\varphi_2$ denote the characteristic functions for
\begin{equation}\label{even-hecke-ops}
K_2
\begin{bmatrix}
\varpi_2 & 0\\
0 & \varpi_2
\end{bmatrix}K_2,~K_2
\begin{bmatrix}
\varpi_2 & 0\\
0 & 1
\end{bmatrix}K_2
\end{equation}
respectively, and $\phi_1,~\phi_2,~\phi_3,~\phi_4$ denote the characteristic functions for
\begin{equation}\label{odd-hecke-ops}
K_p
\begin{bmatrix}
p & & & \\
& p & & \\
& & p & \\
& & & p
\end{bmatrix}K_p,~K_p
\begin{bmatrix}
p & & & \\
& p & & \\
& & p & \\
& & & 1
\end{bmatrix}K_p,
~K_p
\begin{bmatrix}
p & & & \\
& p & & \\
& & 1 & \\
& & & 1
\end{bmatrix}K_p,~K_p
\begin{bmatrix}
p & & & \\
& 1 & & \\
& & 1 & \\
& & & 1
\end{bmatrix}K_p
\end{equation}
respectively when $p\not=2$. Recall that $\varpi_2$ denotes a prime element of $B_2$. We want to obtain the single coset decomposition for the above double cosets. For that, we next review the Bruhat decomposition of $K_p$ given by
\[
K_p=\bigsqcup_{w\in W_p}T_pwT_p,
\]
where $W_p$ denotes the Weyl group of $\GL_2(B_p)$, and $T_p$ the subgroup of elements in $K_p$ which are upper triangular modulo $p$. 

Let $N_p$ be the standard maximal unipotent subgroup of $GL_2(B_p)$ defined over $\Q_p$. 
For this we note that $N_p$ for an odd $p$ is not isomorphic to that for $p=2$. We put $N_p^0(\Z_p):=T_p\cap {}^tN_p(\Q_p)$. We furthermore introduce $N(\Z_p) = N_p(\Q_p)\cap K_p$ and $D(\Z_p) = D_p(\Q_p)\cap K_p$, 
where $D_p$ denotes the subgroup of diagonal matrices in $GL_2(B_p)$. Then we have
\[
T_p=N(\Z_p)D(\Z_p)N_p^0(\Z_p)
\]
(see \cite[Theorem 2.5]{Iwh-M}). Let $h$ be one of 
\[
1_4~(\text{or $1_2$}),~
\begin{bmatrix}
p & & & \\
& p & & \\
& & p & \\
& & & p
\end{bmatrix},~
\begin{bmatrix}
p & & & \\
& p & & \\
& & p & \\
& & & 1
\end{bmatrix},~
\begin{bmatrix}
p & & & \\
& p & & \\
& & 1 & \\
& & & 1
\end{bmatrix},~
\begin{bmatrix}
p & & & \\
& 1 & & \\
& & 1 & \\
& & & 1
\end{bmatrix}~\text{or}~ 
\begin{bmatrix}
\varpi_2 & \\
 & 1
\end{bmatrix}.
\]
\begin{lem}\label{coset-decomp-I}
\[
K_phK_p=\bigsqcup_{w\in W_p/W_p(h)}N(\Z_p)whK_p,
\]
where $W_p(h):=\{w\in W_p\mid whw^{-1}=h\}$. 
\end{lem}
To describe this coset decomposition of $K_phK_p$ explicitly we need a set of representatives for $W_p/W_p(h)$.
\begin{lem}\label{Weyl-gp-decomp}
\begin{enumerate}
\item Let $p=2$. For $h=\begin{bmatrix}
\varpi_2 & 0\\
0 & 1
\end{bmatrix}$ we can take
\[
\left\{1_2,~
\begin{bmatrix}
0 & 1\\
1 & 0
\end{bmatrix}\right\}
\]
as a set of representatives for $W_2/W_2(h)$.
\item Let $p\not=2$.
\begin{enumerate}
\item When $h=\begin{bmatrix}
p & & & \\
& p & & \\
& & p & \\
& & & 1
\end{bmatrix}$ we can take
\[
\left\{
1_4,~
\begin{bmatrix}
0 & 0 & 0 & 1\\
0 & 1 & 0 & 0\\
0 & 0 & 1 & 0\\
1 & 0 & 0 & 0
\end{bmatrix},~
\begin{bmatrix}
1 & 0 & 0 & 0\\
0 & 0 & 0 & 1\\
0 & 0 & 1 & 0\\
0 & 1 & 0 & 0
\end{bmatrix},~
\begin{bmatrix}
1 & 0 & 0 & 0\\
0 & 1 & 0 & 0\\
0 & 0 & 0 & 1\\
0 & 0 & 1 & 0
\end{bmatrix}
\right\}
\]
as a set of representatives for $W_p/W_p(h)$.
\item When $h=\begin{bmatrix}
p & & & \\
& 1 & & \\
& & 1 & \\
& & & 1
\end{bmatrix}$ we can take
\[
\left\{
1_4,~\begin{bmatrix}
0 & 0 & 0 & 1\\
0 & 1 & 0 & 0\\
0 & 0 & 1 & 0\\
1 & 0 & 0 & 0
\end{bmatrix},~
\begin{bmatrix}
0 & 0 & 1 & 0\\
0 & 1 & 0 & 0\\
1 & 0 & 0 & 0\\
0 & 0 & 0 & 1
\end{bmatrix},~
\begin{bmatrix}
0 & 1 & 0 & 0\\
1 & 0 & 0 & 0\\
0 & 0 & 1 & 0\\
0 & 0 & 0 & 1
\end{bmatrix}
\right\}
\]
as a set of representatives for $W_p/W_p(h)$.
\item When $h=\begin{bmatrix}
p & & & \\
& p & & \\
& & 1 & \\
& & & 1
\end{bmatrix}$ we can take
\begin{align*}
&\Big\{1_4,~\begin{bmatrix}
0 & 0 & 1 & 0\\
0 & 1 & 0 & 0\\
1 & 0 & 0 & 0\\
0 & 0 & 0 & 1
\end{bmatrix},~
\begin{bmatrix}
0 & 0 & 0 & 1\\
0 & 1 & 0 & 0\\
0 & 0 & 1 & 0\\
1 & 0 & 0 & 0
\end{bmatrix},~
\begin{bmatrix}
1 & 0 & 0 & 0\\
0 & 0 & 1 & 0\\
0 & 1 & 0 & 0\\
0 & 0 & 0 & 1
\end{bmatrix},~
\begin{bmatrix}
1 & 0 & 0 & 0\\
0 & 0 & 0 & 1\\
0 & 0 & 1 & 0\\
0 & 1 & 0 & 0
\end{bmatrix},\\
& \qquad \qquad 
\begin{bmatrix}
0 & 0 & 1 & 0\\
0 & 0 & 0 & 1\\
1 & 0 & 0 & 0\\
0 & 1 & 0 & 0
\end{bmatrix}
\Big\}
\end{align*}
as a set of representatives for $W_p/W_p(h)$.
\end{enumerate}
\end{enumerate}
\end{lem}
By Lemma \ref{coset-decomp-I} and Lemma \ref{Weyl-gp-decomp} we are now able to write down the coset decomposition of $K_phK_p$ explicitly.
\begin{lem}\label{Gl4-double-cos-lem}
\begin{enumerate}
\item Let $p=2$ and $h=
\begin{bmatrix}
\varpi_2 & 0\\
0 & 1
\end{bmatrix}$. We have
\[
K_2hK_2=
\begin{bmatrix}
1 & 0\\
0 & \varpi_2
\end{bmatrix}K_2\sqcup\bigsqcup_{x\in\mO_2/\varpi_2\mO_2}
\begin{bmatrix}
\varpi_2 & 0\\
0 & 1
\end{bmatrix}
\begin{bmatrix}
1 & \varpi_2^{-1}x\\
0 & 1
\end{bmatrix}K_2.
\]
\item Let $p\not=2$.
\begin{enumerate} 
\item For $h=
\begin{bmatrix}
p & & & \\
& p & & \\
& & p & \\
& & & 1
\end{bmatrix}$ we have
\begin{align*}
K_phK_p&=\bigsqcup_{x_{14},x_{24},x_{34}\in\Z_p/p\Z_p}
\begin{bmatrix}
p & & & \\
& p & & \\
& & p & \\
& & & 1
\end{bmatrix}
\begin{bmatrix}
1 & & & p^{-1}x_{14}\\
& 1 & & p^{-1}x_{24}\\
& & 1 & p^{-1}x_{34}\\
& & & 1
\end{bmatrix}K_p\\
&\sqcup\bigsqcup_{x_{12}\in\Z_p/p\Z_p}
\begin{bmatrix}
p & & & \\
& 1 & & \\
& & p & \\
& & & p
\end{bmatrix}
\begin{bmatrix}
1 & p^{-1}x_{12} & & \\
  & 1 & & \\
  & & 1 & \\
  & & & 1
\end{bmatrix}K_p\\
&\sqcup\bigsqcup_{x_{13},x_{23}\in\Z_p/p\Z_p}
\begin{bmatrix}
p & & & \\
& p & & \\
& & 1 & \\
& & & p
\end{bmatrix}
\begin{bmatrix}
1 & & p^{-1}x_{13} & \\
  & 1 & p^{-1}x_{23} & \\
  & & 1 &\\
  & & & 1
\end{bmatrix}K_p\sqcup
\begin{bmatrix}
1 & & & \\
  & p & & \\
  & & p & \\
  & & & p
\end{bmatrix}K_p.
\end{align*}
\item For $h=
\begin{bmatrix}
p & & & \\
& 1 & & \\
& & 1 & \\
& & & 1
\end{bmatrix}$ we have
\begin{align*}
K_phK_p&=\bigsqcup_{x_{12},x_{13},x_{14}\in\Z_p/p\Z_p}
\begin{bmatrix}
p & & & \\
& 1 & & \\
& & 1 & \\
& & & 1
\end{bmatrix}
\begin{bmatrix}
1 & p^{-1}x_{12} & p^{-1}x_{13} & p^{-1}x_{14} \\
  & 1 & & \\
  & & 1 & \\
  & & & 1
\end{bmatrix}K_p\\
&\sqcup\bigsqcup_{x_{34}\in\Z_p/p\Z_p}
\begin{bmatrix}
1 & & & \\
& 1 & & \\
& & p & \\
& & & 1
\end{bmatrix}
\begin{bmatrix}
1 & & & \\
& 1 & & \\
& & 1 & p^{-1}x_{34} \\
& & & 1
\end{bmatrix}K_p\\
&\sqcup\bigsqcup_{x_{23},x_{24}\in\Z_p/p\Z_p}
\begin{bmatrix}
1 & & & \\
& p & & \\
& & 1 & \\
& & & 1
\end{bmatrix}
\begin{bmatrix}
1 & & & \\
& 1 & p^{-1}x_{23} & p^{-1}x_{24}\\
& & 1 & \\
& & & 1
\end{bmatrix}K_p \\
& \qquad \sqcup
\begin{bmatrix}
1 & & & \\
& 1 & & \\
& & 1 & \\
& & & p
\end{bmatrix}K_p.
\end{align*}
\item For $h=
\begin{bmatrix}
p & & & \\
& p & & \\
& & 1 & \\
& & & 1
\end{bmatrix}$ we have
\begin{align*}
K_phK_p&=\bigsqcup_{x_{13},x_{14},x_{23},x_{24}\in\Z_p/p\Z_p}
\begin{bmatrix}
p & & & \\
& p & & \\
& & 1 & \\
& & & 1
\end{bmatrix}
\begin{bmatrix}
1 & & p^{-1}x_{13} & p^{-1}x_{14}\\
& 1 & p^{-1}x_{23} & p^{-1}x_{24} \\
& & 1 & \\
& & & 1
\end{bmatrix}K_p\\
&\sqcup\bigsqcup_{x_{24},x_{34}\in\Z_p/p\Z_p}
\begin{bmatrix}
1 & & & \\
& p & & \\
& & p & \\
& & & 1
\end{bmatrix}
\begin{bmatrix}
1 & & & \\
& 1 & & p^{-1}x_{24}\\
& & 1 & p^{-1}x_{34} \\
& & & 1
\end{bmatrix}K_p\\
&\sqcup\bigsqcup_{x_{23}\in\Z_p/p\Z_p}
\begin{bmatrix}
1 & & & \\
& p & & \\
& & 1 & \\
& & & p
\end{bmatrix}
\begin{bmatrix}
1 & & & \\
& 1 & p^{-1}x_{23} & \\
& & 1 & \\
& & & 1
\end{bmatrix}K_p\\
&\sqcup\bigsqcup_{x_{12},x_{14},x_{34}\in\Z_p/p\Z_p}
\begin{bmatrix}
p & & & \\
& 1 & & \\
& & p & \\
& & & 1
\end{bmatrix}
\begin{bmatrix}
1 & p^{-1}x_{12} & & p^{-1}x_{14} \\
& 1 & & \\
& & 1 & p^{-1}x_{34} \\
& & & 1
\end{bmatrix}K_p\\
&\sqcup\bigsqcup_{x_{12},x_{13}\in\Z_p/p\Z_p}
\begin{bmatrix}
p & & & \\
& 1 & & \\
& & 1 & \\
& & & p
\end{bmatrix}
\begin{bmatrix}
1 & p^{-1}x_{12} & p^{-1}x_{13} & \\
& 1 & & \\
& & 1 & \\
& & & 1
\end{bmatrix}K_p \\
& \qquad \sqcup
\begin{bmatrix}
1 & & & \\
& 1 & & \\
& & p & \\
& & & p
\end{bmatrix}K_p.
\end{align*}
\end{enumerate}
\end{enumerate}
\end{lem}
We can now describe the actions of Hecke operators defined by $K_phK_p$'s above. With the invariant measure $dx$ of $G_p$ normalized so that $\displaystyle\int_{K_p}dx=1$, we define $K_phK_p\cdot\Phi$ by
\[
(K_phK_p\cdot\Phi)(g):=\displaystyle\int_{G_p}{\ch_{K_phK_p}}(x)\Phi(gx)dx
\]
for $\Phi\in M(\cG(\A),r)$, where $\ch_{K_phK_p}$ denotes the characteristic function for $K_phK_p$. 

We provide the non-adelic description of $K_phK_p\cdot\Phi$, which enables us to describe explicitly the influence of the $K_phK_p$-action on Fourier coefficients of the lifting $F_f$. To this end we need the two following lemmas.
\begin{lem}\label{N_p-invariance}
Let $\Phi\in M(\cG(\A),r)$ be a cusp form. 
For all finite prime $p$, $\Phi$ satisfies
\[
\Phi(n_pg)=\Phi(g)
\]
for any $(n_p,g)\in N_p\times\cG(A)$, where we view $N_p$ as a subgroup of $\cG(\A)$.
\end{lem}
\begin{proof}
The $\cG(\A)$-module generated by $\Phi$ is a finite sum of irreducible cuspidal representations of $\cG(\A)$. As is well-known, an irreducible automorphic representation of $\cG(\A)$ decomposes into a restricted tensor product of irreducible admissible representations of $G_v=\cG(\Q_v)$ over all places $v$. Since $\Phi$ is right $K_p$-invariant for any finite prime $p$, every non-archimedean local component of an irreducible summand for the cuspidal representation generated by $\Phi$ is an unramified principal series representation at an odd prime $p$ and a spherical principal representation at the even prime $p=2$~(see \cite[4.4]{Ca}). The statement is a consequence of the left $N_p$-invariance of such irreducible admissible representations of $G_p$ for $p<\infty$.
\end{proof}

We now introduce the set
\begin{equation}\label{Cp-defn}
C_p:=\{\alpha\in\mO\mid \nu(\alpha)=p\}/\mO^{\times},\quad C'_p:=\{x\in M_2(\Z_p)\mid\det(x)=p\}/GL_2(\Z_p).
\end{equation}
\begin{lem}\label{Cp-lem}
\begin{enumerate}
\item There is a bijection
\[
C'_p\simeq\left\{\left.
\begin{bmatrix}
1 & 0\\
0 & p
\end{bmatrix},~
\begin{bmatrix}
p & b\\
0 & 1
\end{bmatrix}~\right|~b\in\Z/p\Z\right\}.
\]
\item For an odd prime $p$ the isomorphism $\mO_p\simeq M_2(\Z_p)$ induces the bijection
\[
C_p\simeq C'_p.
\]
\end{enumerate}
\end{lem}
\begin{proof}
The first assertion is verified by a direct calculation. We prove the second assertion. 
As is remarked in the proof of \cite[Proposition 5.2]{P} we have $\# C_p=p+1$. 
Under the isomorphism $\mO_p\simeq M_2(\Z_p)$, we can regard any elements in $\mO$ as those in $M_2(\Z_p)$. 
Any two inequivalent representatives of $C_p$ are not equivalent to each other in $\{x\in M_2(\Z_p)\mid\det(x)=p\}/GL_2(\Z_p)$. Otherwise there are two inequivalent representatives $\alpha_1$ and $\alpha_2$ of $C_p$ which are equivalent under $\mO_l^{\times}$-action for all prime $l$, since $\alpha_1/\alpha_2\in\mO_l^{\times}$ for all prime $l$. 
Here note that $\mO^{\times}_l\simeq GL_2(\Z_l)$ for any odd prime $l$. 
This however implies that such two representatives are equivalent to each other in $C_p$. 
We therefore know that there is a injection from $C_p$ into $\{x\in M_2(\Z_p)\mid\det(x)=p\}/GL_2(\Z_p)$. Since the latter set also has $p+1$ representatives as in the statement the injection is actually a bijection.
\end{proof}

Let $F\in M(\GL_2(\mO),r)$ correspond to $\Phi$. By $K_phK_p\cdot F$ we denote the cusp form in $M(\GL_2(\mO),r)$ corresponding to $K_phK_p\cdot\Phi$. Due to Lemma \ref{cusp-lem}, $\cG(\Q)\backslash\cG(\A)/U$ has a complete set of representatives in $\cG(\R)$, where $\cG(\R)$ is viewed as a subgroup of $\cG(\A)$ in the usual manner. 
The cusp form $K_phK_p\cdot\Phi$ is thereby determined by its restriction to $\cG(\R)$, which is nothing but $K_phK_p\cdot F$. 
Moreover, we remark that an element in $M(\GL_2(\mO),r)$ is determined by its restriction 
to $NA=\{n(x)a_y\mid x\in\Hl,~y\in\R_+^{\times}\}$~(see \eqref{Iwasawa-decomp}).
\begin{prop}
Let $K_phK_p\cdot F$ be as above and let $n(x)$ and $a_y$ be as defined in \eqref{Iwasawa-decomp}. For $p$ odd, let $C_p$ be as defined in \eqref{Cp-defn}.
\begin{enumerate}
\item Let $p=2$. We have
\[
(K_2hK_2\cdot F)(n(x)a_y)=F(n(\varpi_2 x)a_{2^{\frac{1}{2}}y})+2^2F(n(\varpi_2^{-1}x)a_{2^{-\frac{1}{2}}y}).
\]
\item Let $p\not=2$.
\begin{enumerate}
\item When $h=\begin{bmatrix}
p & & & \\
& p & & \\
& & p & \\
& & & 1
\end{bmatrix}$ we have
\[
(K_phK_p\cdot F)(n(x)a_y)=\sum_{\alpha\in C_p}F(n(\alpha x)a_{p^{\frac{1}{2}}y})+p^2\sum_{\alpha\in C_p}F(n(x\alpha^{-1})a_{p^{-\frac{1}{2}}y}).
\]
\item When $h=\begin{bmatrix}
p & & & \\
& 1 & & \\
& & 1 & \\
& & & 1
\end{bmatrix}$ we have
\[
(K_phK_p\cdot F)(n(x)a_y)=\sum_{\alpha\in C_p}F(n(x\alpha)a_{p^{\frac{1}{2}}y})+p^2\sum_{\alpha\in C_p}F(n(\alpha^{-1}x)a_{p^{-\frac{1}{2}}y}).
\]
\item When $h=\begin{bmatrix}
p & & & \\
& p & & \\
& & 1 & \\
& & & 1
\end{bmatrix}$ we have
\begin{align*}
 (K_phK_p\cdot F)(n(x)a_y)&=F(n(px)a_{py})+p^4F(n(p^{-1}x)a_{p^{-1}y}) \\
 & \qquad \qquad +p\sum_{(\alpha_1,\alpha_2)\in C_p\times C_p}F(n(\alpha_1^{-1}x\alpha_2)a_y).
\end{align*}
\end{enumerate}
\end{enumerate}
\end{prop}
\begin{proof}
We prove only 2 (c). The other cases are settled similarly. 
The left coset decomposition of $K_phK_p$ in part (2) (iii) of Lemma \ref{Gl4-double-cos-lem} can be rewritten as
\begin{align*}
&\bigsqcup_{x_{13},x_{14},x_{23},x_{24}\in\Z_p/p\Z_p}
\begin{bmatrix}
1 & & x_{13} & x_{14}\\
& 1 & x_{23} & x_{24} \\
& & 1 & \\
& & & 1
\end{bmatrix}
\begin{bmatrix}
p & & & \\
& p & & \\
& & 1 & \\
& & & 1
\end{bmatrix}
K_p\\
&\sqcup\bigsqcup_{x_{24},x_{34}\in\Z_p/p\Z_p}
\begin{bmatrix}
1 & & & \\
& 1 & & x_{24}\\
& & 1 & \\
& & & 1
\end{bmatrix}
\begin{bmatrix}
1 & & & \\
& p & & \\
& & p & x_{34} \\
& & & 1
\end{bmatrix}K_p\\
&\sqcup\bigsqcup_{x_{23}\in\Z_p/p\Z_p}
\begin{bmatrix}
1 & & & \\
& 1 & x_{23} & \\
& & 1 & \\
& & & 1
\end{bmatrix}
\begin{bmatrix}
1 & & & \\
& p & & \\
& & 1 & \\
& & & p
\end{bmatrix}K_p\\
&\sqcup\bigsqcup_{x_{12},x_{14},x_{34}\in\Z_p/p\Z_p}
\begin{bmatrix}
1 & & & x_{14} \\
& 1 & & \\
& & 1 & \\
& & & 1
\end{bmatrix}
\begin{bmatrix}
p & x_{12} & & \\
& 1 & & \\
& & p & x_{34} \\
& & & 1
\end{bmatrix}K_p\\
&\sqcup\bigsqcup_{x_{12},x_{13}\in\Z_p/p\Z_p}
\begin{bmatrix}
1 & & x_{13} & \\
& 1 & & \\
& & 1 & \\
& & & 1
\end{bmatrix}
\begin{bmatrix}
p & x_{12} &  & \\
& 1 & & \\
& & 1 & \\
& & & p
\end{bmatrix}K_p
\sqcup
\begin{bmatrix}
1 & & & \\
& 1 & & \\
& & p & \\
& & & p
\end{bmatrix}K_p.
\end{align*} 
We regard the identity $1$~(respectively~the zero $0$) of $B$ as the identity $1_v$~(respectively~$0_v$) of $B_v$ for $v\le\infty$. 
Taking Lemma \ref{N_p-invariance} into consideration, the contribution of the first and the last left cosets to $K_phK_p\cdot\Phi(g)$ can be written as
\[
p^4\Phi((\prod_{v\le\infty,~v\not=p}
\begin{bmatrix}
1_v & 0_v\\
0_v & 1_v
\end{bmatrix}\times
\begin{bmatrix}
p1_p & 0_p\\
0_p & 1_p
\end{bmatrix})g)+\Phi((\prod_{v\le\infty,~v\not=p}
\begin{bmatrix}
1_v & 0_v\\
0_v & 1_v
\end{bmatrix}\times
\begin{bmatrix}
1_p & 0_p\\
0_p & p1_p
\end{bmatrix})g)
\]
and that of the remaining four cosets to $K_phK_p\cdot\Phi(g)$ as
\[
p\sum_{(\alpha'_1,\alpha'_2)\in C'_p\times C'_p}\Phi((\prod_{v\le\infty,~v\not=p}
\begin{bmatrix}
1_v & 0_v\\
0_v & 1_v
\end{bmatrix}\times
\begin{bmatrix}
\alpha'_1 & 0_p\\
0_p & \alpha'_2
\end{bmatrix})g)
\]
for $g\in\cG(\R)\subset\cG(\A)$. 
 
In $\cG(\Q)\backslash\cG(\A)/U$, 
\begin{align*}
\prod_{v\le\infty,~v\not=p}
\begin{bmatrix}
1_v & 0_v\\
0_v & 1_v
\end{bmatrix}\times
\begin{bmatrix}
p1_p & 0_p\\
0_p & 1_p
\end{bmatrix} &= \prod_{v<\infty}
\begin{bmatrix}
1_v & 0_v\\
0_v & 1_v
\end{bmatrix}
\times
\begin{bmatrix}
p^{-1}1_{\infty} & 0_{\infty}\\
0_{\infty} & 1_{\infty}
\end{bmatrix},\\
\prod_{v\le\infty,~v\not=p}
\begin{bmatrix}
1_v & 0_v\\
0_v & 1_v
\end{bmatrix}\times
\begin{bmatrix}
1_p & 0_p\\
0_p & p1_p
\end{bmatrix} &= \prod_{v<\infty}
\begin{bmatrix}
1_v & 0_v\\
0_v & 1_v
\end{bmatrix}\times
\begin{bmatrix}
1_{\infty} & 0_{\infty}\\
0_{\infty} & p^{-1}1_{\infty}
\end{bmatrix}.
\end{align*} 
By Lemma \ref{Cp-lem} the cosets represented by $\prod_{v\le\infty,~v\not=p}
\begin{bmatrix}
1_v & 0_v\\
0_v & 1_v
\end{bmatrix}\times
\begin{bmatrix}
\alpha'_1 & 0_p\\
0_p & \alpha'_2
\end{bmatrix}$ with  $\alpha'_1,\alpha'_2 \in C'_p$ are in bijection with the cosets in $\GL_2(\mO)\backslash\cG(\R)$ represented by 
$\{\begin{bmatrix}
\alpha_1^{-1} & 0\\
0 & \alpha_2^{-1}
\end{bmatrix} \mid (\alpha_1,\alpha_2)\in C_p\times C_p\}$. 

Now let us put $g:=\prod_{v<\infty}
\begin{bmatrix}
1_v & 0_v\\
0_v & 1_v
\end{bmatrix}\times n(x)a_y\in\cG(\R)\subset\cG(\A)$. 
We therefore have
\begin{align*}
&p^4\Phi((\prod_{v\le\infty,~v\not=p}
\begin{bmatrix}
1_v & 0_v\\
0_v & 1_v
\end{bmatrix}\times
\begin{bmatrix}
p1_p & 0_p\\
0_p & 1_p
\end{bmatrix})g)+\Phi((\prod_{v\le\infty,~v\not=p}
\begin{bmatrix}
1_v & 0_v\\
0_v & 1_v
\end{bmatrix}\times
\begin{bmatrix}
1_p & 0_p\\
0_p & p1_p
\end{bmatrix})g)\\
&=p^4F(
\begin{bmatrix}
p^{-1} & 0\\
0 & 1
\end{bmatrix}n(x)a_y)+F(
\begin{bmatrix}
1 & 0\\
0 & p^{-1}
\end{bmatrix}n(x)a_y)
\end{align*}
and
\[
p\sum_{(\alpha'_1,\alpha'_2)\in C'_p\times C'_p}\Phi((\prod_{v\le\infty,~v\not=p}
\begin{bmatrix}
1_v & 0_v\\
0_v & 1_v
\end{bmatrix}\times
\begin{bmatrix}
\alpha'_1 & 0_p\\
0_p & \alpha'_2
\end{bmatrix})g)=p\sum_{(\alpha_1,\alpha_2)\in C_p\times C_p}F(
\begin{bmatrix}
\alpha_1^{-1} & 0\\
0 & \alpha_2^{-1}
\end{bmatrix}n(x)a_y)).
\]
Noting the invariance of $F$ with respect to $K$ and $Z^+$~(see \eqref{Iwasawa-decomp} for $Z^+$), we deduce the assertion from this by a direct computation.
\end{proof}

For this proposition we remark that the formulas above do not depend on the choices of representatives of $C_p$ since $F$ is left and right invariant with respect to $\{
\begin{bmatrix}
u_1 & \\
    & u_2
\end{bmatrix}\mid u_1,~u_2\in\mO^{\times}\}$. Let the Fourier decomposition of $(K_phK_p\cdot F)$ be given by
$$(K_phK_p\cdot F)(n(x)a_y) = \sum\limits_{\beta \in S \backslash \{0\}} (K_phK_p\cdot F)_{\beta} y^2K_{\sqrt{-1}r}(2\pi|\beta|y)e^{2\pi\sqrt{-1}\real(\beta x)}.$$
The next proposition provides a formula for $(K_phK_p\cdot F)_{\beta}$ in terms of the Fourier coefficients $A(\beta)$ of $F$.
\begin{prop}\label{hecke-fourier-prop}
\begin{enumerate}
\item Let $p=2$. We obtain
\[
(K_2hK_2\cdot F)_{\beta}=2(A(\beta\varpi_2^{-1})+A(\beta\varpi_2)).
\]

\item
Let $p$ be an odd prime and $\beta\in S\setminus\{0\}$.
\begin{enumerate}
\item When $h=\begin{bmatrix}
p & & & \\
& p & & \\
& & p & \\
& & & 1
\end{bmatrix}$, 
\[
(K_phK_p\cdot F)_{\beta}=p(\sum_{\alpha\in C_p}A(\beta\bar{\alpha}^{-1})+\sum_{\alpha\in C_p}A(\bar{\alpha}\beta)).
\]
\item When $h=\begin{bmatrix}
p & & & \\
& 1 & & \\
& & 1 & \\
& & & 1
\end{bmatrix}$, 
\[
(K_phK_p\cdot F)_{\beta}=p(\sum_{\alpha\in C_p}A(\alpha^{-1}\beta)+\sum_{\alpha\in C_p}A(\beta\alpha)).
\]
\item When $h=\begin{bmatrix}
p & & & \\
& p & & \\
& & 1 & \\
& & & 1
\end{bmatrix}$, 
$$ (K_phK_p\cdot F)_{\beta} =(p^{2}A(p^{-1}\beta)+p^{2}A(p\beta) 
+p\sum_{(\alpha_1,\alpha_2)\in C_p\times C_p}A(\alpha_{1}^{-1}\beta\alpha_2)).
$$
\end{enumerate}
\end{enumerate}
\end{prop}
For this proposition we note that the automorphy of $F$ with respect to $\{
\begin{bmatrix}
u_1 & \\
    & u_2
\end{bmatrix}\mid u_1,~u_2\in\mO^{\times}\}$ implies $A(u_1\beta u_2)=A(\beta)$ for $\beta\in S\setminus\{0\}$ and $u_1,~u_2\in\mO^{\times}$. 
From this we see that the formulas in 2)(b) and 2)(c) do not depend on the choices of representatives for $C_p$. 
As for 2)(a) we furthermore see that, given any complete set $\{\alpha_i\mid 1\le i\le p+1\}$ of representatives for $C_p$, $\{\bar{\alpha_i}\mid 1\le i\le p+1\}$ also forms such a set. As a result we see that the formula in 2)(a) is also not dependent on the choices of representatives for $C_p$.

\subsection{Hecke equivariance for $p=2$}\label{Hecke-equiv-even}
Let $f \in S(\Gamma_0(2);-(\frac{1}{4}+\frac{r^2}{4}))$ be a new form~(for the definition see \cite[Section 8.5]{Iwn}) with Hecke eigenvalue $\lambda_p$ for $p=2$. 
By the Hecke eigenvalue $\lambda_2$ we mean the eigenvalue of $f$ for the $U(2)$ operator defined by the action of the double coset $\Gamma_0(2) \mat{1}{}{}{2} \Gamma_0(2)$. Let us also assume that $f$ is an eigenfunction of the Atkin Lehner involution with eigenvalue $\epsilon$. It can be checked that $\lambda_2$ and $\epsilon$ are related by
\begin{equation}\label{lambda2-epsilon-reln}
\lambda_2 = -\epsilon.
\end{equation}
Using the single coset decomposition 
$$\Gamma_0(2) \mat{1}{}{}{2} \Gamma_0(2) = \Gamma_0(2) \mat{1}{1}{}{2} \sqcup \Gamma_0(2) \mat{1}{}{}{2},$$
we get 
$$f(\frac{z+1}2) + f(\frac z2) = \lambda_2 f(z).$$
In terms of Fourier coefficients of $f$, using (\ref{lambda2-epsilon-reln}), we get 
\begin{equation}\label{f-hecke-reln-p=2}
c(2m) = \frac{\lambda_2}2 c(m) = -\frac{\epsilon}2 c(m), \text{ for all } m \in \Z.
\end{equation}

\begin{prop}\label{p=2-hecke-equiv-prop}
Let $f \in S(\Gamma_0(2);-(\frac{1}{4}+\frac{r^2}{4}))$ be a new form with Hecke eigenvalue $\lambda_p$ for $p=2$ and an eigenfunction of the Atkin Lehner involution with eigenvalue $\epsilon$. Let $F = F_f$ be as defined in Theorem \ref{lift-thm}. Then
\begin{equation}\label{p=2-hecke-equiv-eqn}
(K_2 \mat{\varpi_2}{}{}{1} K_2) F = -3\sqrt{2}\epsilon F.
\end{equation}
\end{prop}
\begin{proof}
Let $\beta = \varpi_2^u d \beta_0$ be a decomposition according to Proposition \ref{main-prop}. Hence, $u \geq 0, d$ is odd and $\beta_0 \in S^{\rm prim}$. Using (\ref{lambda2-epsilon-reln}) and (\ref{f-hecke-reln-p=2}) we see that 
\begin{align*}
A(\beta) &= (2^{u+1}-1) |\beta| \sum\limits_{n | d} c\Big(\frac{-|\beta|^2}{2n^2}\Big) \\
A(\beta \varpi_2) &= (2^{u+2}-1) \frac{-\epsilon}{\sqrt{2}} |\beta| \sum\limits_{n | d} c\Big(\frac{-|\beta|^2}{2n^2}\Big) \\
A(\beta \varpi_2^{-1}) &= (2^u-1) (-\epsilon \sqrt{2}) |\beta| \sum\limits_{n | d} c\Big(\frac{-|\beta|^2}{2n^2}\Big).
\end{align*}
Note that, if $u=0$, then $A(\beta \varpi_2^{-1}) = 0$ and so is the right hand side of the third equality above. We have
$$\frac{2^{u+2}-1}{\sqrt{2}} + (2^u-1)\sqrt{2} = \frac 3{\sqrt{2}} (2^{u+1}-1).$$
Hence, we have
$$2\Big(A(\beta \varpi_2) + A(\beta \varpi_2^{-1})\Big) = -3\sqrt{2}\epsilon A(\beta).$$
The proposition now follows from part 1. of Proposition \ref{hecke-fourier-prop}.
\end{proof}

\subsection{Hecke equivariance for odd primes}
We assume that $f \in S(\Gamma_0(2);-(\frac{1}{4}+\frac{r^2}{4}))$ is a Hecke eigenform with Hecke eigenvalue $\lambda_p$ for every odd prime $p$ but do not assume that $f$ is a new form. In terms of Fourier coefficients of $f$, the Hecke relation is given by 
\begin{equation}\label{oddp-hecke-reln-f}
p^{\frac 12} c(pn) + p^{-\frac 12} c(n/p) = \lambda_p c(n),
\end{equation}
where $c(n/p)$ is assumed to be zero if $p$ does not divide $n$. The following lemma will play a key role in the computation of the Hecke operator.
\begin{lem}\label{fund-lemma}
Let $\beta \in S^{\rm prim}$. Then
\begin{equation}\label{fund-eqn}
\# \{ \alpha \in C_p : p | \beta \alpha \} = \# \{ \alpha \in C_p : p | \alpha \beta \} = \begin{cases} 1 & \text{ if } p\, | \,|\beta|^2, \\
0 & \text{ if } p\not|\, |\beta|^2.
\end{cases}
\end{equation}
In addition, $p^2$ does not divide $\alpha \beta$ or $\beta \alpha$ for any $\alpha \in C_p$.
\end{lem}
\begin{proof}
Note that, by taking conjugates, it is enough to prove the statement of the lemma for $\{\alpha \in C_p : p | \beta \alpha\}$ for all $\beta$. Taking norms, it is clear that $p$ does not divide $\beta \alpha$ if $p$ does not divide $|\beta|^2$. Hence, assume that $p$ divides $|\beta|^2$.  Let $\beta = \beta_1 + \beta_2 i + \beta_3 j + \beta_4 ij$. The conditions $p | |\beta|^2$ and $p \not| \,\beta$, imply that there is a pair amongst the set $\{\beta_1, \beta_2, \beta_3, \beta_4\}$ which does not satisfy $x^2 + y^2 \equiv 0 \pmod{p}$. From the proof it will be clear that we can take, without loss of generality, $\beta_3^2 + \beta_4^2 \not\equiv 0 \pmod{p}$. Let $\alpha = \alpha_1 + \alpha_2 i + \alpha_3 j + \alpha_4 ij$. The condition $p | \beta \alpha$ is equivalent to the following matrix equation modulo $p$.
$$\underbrace{\begin{bmatrix}\beta_1 & - \beta_2 & -\beta_3 & -\beta_4 \\
\beta_2 &  \beta_1 & -\beta_4 & \beta_3 \\
\beta_3 &  \beta_4 & \beta_1 & -\beta_2 \\
\beta_4 & - \beta_3 & \beta_2 & \beta_1 \end{bmatrix}}_{P_\beta} \begin{bmatrix} \alpha_1\\\alpha_2\\\alpha_3\\\alpha_4\end{bmatrix} = \begin{bmatrix}0\\0\\0\\0\end{bmatrix}.$$
The matrix $P_\beta$ considered over $\Z_p$ has rank $2$. By our assumption $\beta_3^2 + \beta_4^2 \not\equiv 0 \pmod{p}$, we see that the kernel of $P_\beta$ is spanned by the first two rows of $P_\beta$. Hence, $\alpha$ is given by 
$$\alpha_1 = a\beta_1 + b \beta_2, \quad \alpha_2 = -a\beta_2+b\beta_1,\quad  \alpha_3 = -a\beta_3-b\beta_4, \quad \alpha_4 = -a\beta_4+b\beta_3,$$
for some $a, b \in \Z_p$. This gives us $\alpha_3^2 + \alpha_4^2 = (a^2+b^2)(\beta_3^2 + \beta_4^2) \neq 0$. This is because, by assumption $\beta_3^2 + \beta_4^2 \neq 0$, and $|\alpha|^2 = (a^2+b^2) |\beta|^2$ and $p^2$ does not divide $|\alpha|^2$. 

On Pg 69 of \cite{P1}, it has been shown that the set $S_1 = \{\alpha \in C_p : \alpha_3^2 + \alpha_4^2 \not\equiv 0 \pmod{p} \}$ is in bijection with the set $S_2 = \{(x, y) \in \Z_p \times \Z_p : x^2 + y^2 + 1 \equiv 0 \pmod{p} \}$. The map from $S_1$ to $S_2$ is given as follows. For $\alpha \in S_1$, we obtain $(x_\alpha, y_\alpha)$ as the solution to the matrix equation modulo $p$ given by
$$\mat{\alpha_3}{\alpha_4}{-\alpha_4}{\alpha_3} \begin{bmatrix}x_\alpha \\ y_\alpha \end{bmatrix} = \begin{bmatrix}\alpha_1 \\ \alpha_2 \end{bmatrix}.$$
 One can check that $(x, y) \in S_2$ for the following choice of $x$ and $y$.
\begin{equation}\label{xychoice}
x = \frac{-\beta_1 \beta_3 - \beta_2 \beta_4}{\beta_3^2+\beta_4^2}, \qquad y = \frac{\beta_2 \beta_3 - \beta_4 \beta_1}{\beta_3^2+\beta_4^2}.
\end{equation}
If one takes $\alpha \in S_1$ to be the pre-image of the above $(x,y)$, then we can check that $\alpha \in {\rm Ker}(P_\beta)$. On the other hand, if $\alpha \in C_p$ belongs to ${\rm Ker}(P_\beta)$, then it can also be checked that the corresponding $(x_\alpha, y_\alpha)$ are equivalent modulo $p$ to those in (\ref{xychoice}). This completes the proof of \eqref{fund-eqn}. 

If $p^2$ divides $\beta \alpha$ or $\alpha \beta$ for some $\alpha \in C_p$ then it is clear that $\beta$ cannot be primitive. This completes the proof of the lemma.
\end{proof}

\begin{prop}\label{oddp-hecke-1-2-prop}
Let $f \in S(\Gamma_0(2);-(\frac{1}{4}+\frac{r^2}{4}))$ be a  Hecke eigenform with Hecke eigenvalue $\lambda_p$ for every odd prime $p$. Let $F = F_f$ be as defined in Theorem \ref{lift-thm}. For an odd prime $p$ we then have
\begin{equation}\label{oddp-hecke-1-2-eqn}
(K_p \begin{bmatrix}p\\&p\\&&p\\&&&1\end{bmatrix} K_p) F = (K_p \begin{bmatrix}p\\&1\\&&1\\&&&1\end{bmatrix} K_p) F = p(p+1) \lambda_p F.
\end{equation}
\end{prop}
\begin{proof}
Using Proposition \ref{hecke-fourier-prop} we can show that, if the Fourier coefficients satisfy $A(\beta) = A(\bar\beta)$ for all $\beta \in S$ and the second equality in (\ref{oddp-hecke-1-2-eqn}) holds, then so does the first equality. Since, the Fourier coefficients of $F = F_f$ satisfy the above condition, we are reduced to showing the second equality in (\ref{oddp-hecke-1-2-eqn}).

We will compute the action of the Hecke operator on the Fourier coefficients $A(\beta)$ of $F$. Since, all the computations only involve the prime $p$, it will be enough to consider the case $\beta = p^s \beta_0$ with $s \geq 0$ and $\beta_0 \in S^{\rm prim}$. For such a $\beta$, we have
$$A(\beta) = p^s |\beta_0| \sum\limits_{k=0}^s c\Big(\frac{-|\beta_0|^2p^{2s-2k}}2\Big).$$
Note that $\alpha^{-1}\beta_0 = \frac{1}{p}\bar\alpha\beta_0$ for $\alpha\in B$ with $\nu(\alpha)=p$. Hence, for such $\alpha$, $\alpha^{-1}\beta_0 \in S$ if and only if $p$ divides $\bar\alpha\beta_0$. 

Let us first consider the case where $p$ does not divide $|\beta_0|^2$. Hence, by Lemma \ref{fund-lemma}, we see that $\alpha^{-1}\beta_0 \not\in S$ and $p$ does not divide $\beta_0 \alpha$ for any $\alpha \in C_p$. Hence, for any $\alpha \in C_p$, we have
\begin{equation}\label{1steqn}
A(\beta \alpha) = \sqrt{p} p^s |\beta_0| \sum\limits_{k=0}^s c\Big(\frac{-|\beta_0|^2p^{2s+1-2k}}2\Big)
\end{equation}
and 
\begin{equation}\label{2ndeqn}
A(\alpha^{-1}\beta) =\frac 1{\sqrt{p}} p^s |\beta_0| \sum\limits_{k=0}^{s-1} c\Big(\frac{-|\beta_0|^2p^{2s-1-2k}}2\Big).
\end{equation}
Note that, if $s = 0$, both the left and right hand side of the last equation are zero. Now, using (\ref{oddp-hecke-reln-f}), we get for any $\alpha \in C_p$,
\begin{align*}
 A(\alpha^{-1}\beta) + A(\beta \alpha) &= p^s |\beta_0| \Big(\sum\limits_{k=0}^{s-1} \Big(p^{-1/2} c\Big(\frac{-|\beta_0|^2p^{2s-1-2k}}2\Big) + p^{1/2} c\Big(\frac{-|\beta_0|^2p^{2s+1-2k}}2\Big)\Big) \\
& \qquad  + p^{1/2} c\Big(\frac{-|\beta_0|^2p}2\Big) \Big) \\
&= p^s |\beta_0| \Big(\sum\limits_{k=0}^{s-1} \lambda_p c\Big(\frac{-|\beta_0|^2p^{2s-2k}}2\Big) + \lambda_p c\Big(\frac{-|\beta_0|^2}2\Big)\Big)\\
&= \lambda_p A(\beta).
\end{align*}
Now, using the fact that the number of elements in $C_p$ is $p+1$ and part 2 b) of Proposition \ref{hecke-fourier-prop}, we get the result.

Next, let us assume that $p$ divides $|\beta_0|^2$. By Lemma \ref{fund-lemma}, there is a unique $\alpha_1 \in C_p$ such that $p$ divides $\beta_0 \alpha_1$ and a unique $\alpha_2 \in C_p$ such that $\alpha_2^{-1} \beta_0 \in S$. If $\alpha \in C_p$ but $\alpha \neq \alpha_1$ then the formula for $A(\beta \alpha)$ is the same as in (\ref{1steqn}). For $\alpha = \alpha_1$, we have
$$A(\beta \alpha_1) = \sqrt{p} p^s |\beta_0| \sum\limits_{k=0}^{s+1} c\Big(\frac{-|\beta_0|^2p^{2s+1-2k}}2\Big).$$
For $\alpha \in C_p$ but $\alpha \neq \alpha_2$, the formula for $A(\alpha^{-1}\beta)$ is the same as in (\ref{2ndeqn}). For $\alpha = \alpha_2$, we have
$$A(\alpha_2^{-1}\beta) = \frac 1{\sqrt{p}} p^s |\beta_0| \sum\limits_{k=0}^{s} c\Big(\frac{-|\beta_0|^2p^{2s-1-2k}}2\Big).$$
Hence, we get the following,
\begin{align*}
& \sum\limits_{\alpha \in C_p} (A(\alpha^{-1}\beta) + A(\beta \alpha)) \\
& \qquad = p p^s |\beta_0| \Big[\sum\limits_{k=0}^s p^{1/2} c\Big(\frac{-|\beta_0|^2p^{2s+1-2k}}2\Big) + \sum\limits_{k=0}^{s-1} p^{-1/2} c\Big(\frac{-|\beta_0|^2p^{2s-1-2k}}2\Big)\Big] \\
& \qquad \qquad + p^s |\beta_0| \Big[\sum\limits_{k=0}^{s+1} p^{1/2} c\Big(\frac{-|\beta_0|^2p^{2s+1-2k}}2\Big) + \sum\limits_{k=0}^s p^{-1/2} c\Big(\frac{-|\beta_0|^2p^{2s-1-2k}}2\Big)\Big] \\
& \qquad = p p^s |\beta_0| \Big[\sum\limits_{k=0}^{s-1} \lambda_p c\Big(\frac{-|\beta_0|^2p^{2s-2k}}2\Big) + p^{1/2} c\Big(\frac{-|\beta_0|^2p}2\Big)\Big] \\
& \qquad \qquad + p^s |\beta_0| \Big[\sum\limits_{k=0}^s \lambda_p c\Big(\frac{-|\beta_0|^2p^{2s-2k}}2\Big) + p^{1/2} c\Big(\frac{-|\beta_0|^2}{2p}\Big)\Big] \\
& \qquad = \lambda_p A(\beta) + p p^s |\beta_0| \Big[\sum\limits_{k=0}^{s-1} \lambda_p c\Big(\frac{-|\beta_0|^2p^{2s-2k}}2\Big) + p^{1/2} c\Big(\frac{-|\beta_0|^2p}2\Big) + p^{-1/2} c\Big(\frac{-|\beta_0|^2}{2p}\Big)\Big] \\
& \qquad = \lambda_p A(\beta) + p \lambda_p A(\beta) = (p+1) \lambda_p A(\beta).
\end{align*}
This completes the proof of the proposition.
\end{proof}

\begin{prop}\label{oddp-hecke-3-prop}
Let $f \in S(\Gamma_0(2);-(\frac{1}{4}+\frac{r^2}{4}))$ be a Hecke eigenform with Hecke eigenvalue $\lambda_p$ for every odd prime $p$. Let $F = F_f$ be as defined in Theorem \ref{lift-thm}. For an odd prime $p$ we then have
\begin{equation}\label{oddp-hecke-3-eqn}
(K_p \begin{bmatrix}p\\&p\\&&1\\&&&1\end{bmatrix} K_p) F = \big(p^2 \lambda^2_p + p^3+p\big) F.
\end{equation}
\end{prop}
\begin{proof}
First observe that, using \eqref{oddp-hecke-reln-f}, one can show that, for all $n$,
\begin{equation}\label{lambda2-eq1}
pc(np^2) = (\lambda_p^2-1) c(n) - p^{-1/2} \lambda_p c(n/p).
\end{equation}
If we assume that $p | n$, then we can get another identity given by
\begin{equation}\label{lambda2-eq2}
pc(np^2) + p^{-1}c(n/p^2) = (\lambda_p^2-2) c(n).
\end{equation}
As in the proof of Proposition \ref{oddp-hecke-1-2-prop}, we can assume that $\beta = p^s \beta_0$, where $s \geq 0$ and $\beta_0 \in S^{\rm prim}$. Let us abbreviate $\nu_p(\beta) = s$. For such a $\beta$ we have 
$$A(\beta) = p^s |\beta_0| \sum\limits_{k=0}^s c\Big(\frac{-|\beta_0|^2p^{2s-2k}}2\Big).$$
Hence,
$$A(p \beta) = p^{s+1} |\beta_0| \sum\limits_{k=0}^{s+1} c\Big(\frac{-|\beta_0|^2p^{2s-2k+2}}2\Big)$$
and
$$A(p^{-1}\beta) = p^{s-1} |\beta_0| \sum\limits_{k=0}^{s-1} c\Big(\frac{-|\beta_0|^2p^{2s-2k-2}}2\Big).$$
Next, we need to compute $\sum A(\alpha_1^{-1} \beta \alpha_2)$ where the sum is over all $\alpha_1, \alpha_2$ in $C_p$. We consider three cases depending on whether $p \not| \,|\beta_0|^2$ or $p |\, |\beta_0|^2$ but $p^2 \not| \,|\beta_0|^2$ or $p^2 |\, |\beta_0|^2$.  

{\it Case 1}: Let us assume that $p \not| \, |\beta_0|^2$. Applying Lemma \ref{fund-lemma} to $\beta_0$, we see that $\beta_0 \alpha_2 \in S^{\rm prim}$ for all $\alpha_2 \in C_p$. Again applying Lemma \ref{fund-lemma} to $\beta_0 \alpha_2$ for a fixed $\alpha_2$, we see there is a unique $\alpha_{1,2}\in C_p$ such that $\nu_p(\alpha_{1,2}^{-1} \beta \alpha_2) = s$ and, for all $\alpha_1 \neq \alpha_{1,2}$, we have $\nu_p(\alpha_1^{-1} \beta \alpha_2) = s-1$. Hence,
\begin{align*}
\sum\limits_{\alpha_1, \alpha_2 \in C_p} A(\alpha_1^{-1} \beta \alpha_2) &= \sum\limits_{\alpha_2 \in C_p} \Big(A(\alpha_{1,2}^{-1} \beta \alpha_2) + \sum\limits_{\substack{\alpha_1 \in C_p \\ \alpha_1 \neq \alpha_{1,2}}} A(\alpha_1^{-1} \beta \alpha_2) \Big) \\
&= (p+1) A(\beta) + (p+1) p p^s |\beta_0| \sum\limits_{k=0}^{s-1} c\Big(\frac{-|\beta_0|^2p^{2s-2k}}2\Big).
\end{align*}
Putting this all together, we see that $p^2(A(p\beta) + A(p^{-1}\beta)) + p\sum_{\alpha_1, \alpha_2} A(\alpha_1^{-1} \beta \alpha_2)$ is equal to 
\begin{align*}
& p^2 \Big(p^{s+1} |\beta_0| \sum\limits_{k=0}^{s+1} c\Big(\frac{-|\beta_0|^2p^{2s-2k+2}}2\Big) + p^{s-1} |\beta_0| \sum\limits_{k=0}^{s-1} c\Big(\frac{-|\beta_0|^2p^{2s-2k-2}}2\Big) \Big) \\
& \qquad + p(p+1) A(\beta) + (p+1) p^2 p^s |\beta_0| \sum\limits_{k=0}^{s-1} c\Big(\frac{-|\beta_0|^2p^{2s-2k}}2\Big) \\
=& p^2 p^s |\beta_0| \Big( (\lambda_p^2-2) \sum\limits_{k=0}^{s-1} c\Big(\frac{-|\beta_0|^2p^{2s-2k}}2\Big) + p c\Big(\frac{-|\beta_0|^2p^2}2\Big) + p c\Big(\frac{-|\beta_0|^2}2\Big)\Big) \\
& \qquad + p(p+1) A(\beta) + (p+1) p^2 p^s |\beta_0| \sum\limits_{k=0}^{s-1} c\Big(\frac{-|\beta_0|^2p^{2s-2k}}2\Big) \\
=& p^2 p^s |\beta_0| \Big( (\lambda_p^2-2) \sum\limits_{k=0}^{s} c\Big(\frac{-|\beta_0|^2p^{2s-2k}}2\Big) - (\lambda_p^2-2) c\Big(\frac{-|\beta_0|^2}2\Big) + (\lambda_p^2-1) c\Big(\frac{-|\beta_0|^2}2\Big) \\
& \qquad + p c\Big(\frac{-|\beta_0|^2}2\Big)\Big) + p(p+1) A(\beta) + (p+1) p^2 A(\beta) - (p+1) p^2 p^s |\beta_0| c\Big(\frac{-|\beta_0|^2}2\Big) \\
=& \big(p^2(\lambda_p^2-2) + p(p+1) + p^2(p+1)\big) A(\beta) \\
=& \big(p^2 \lambda_p^2 + p^3 + p\big) A(\beta).
\end{align*}
Here, we have used both \eqref{lambda2-eq1} and \eqref{lambda2-eq2}.

{\it Case 2}: Let $p | \, |\beta_0|^2$ but $p^2 \not| \, |\beta_0|^2$. Applying Lemma \ref{fund-lemma} to $\beta_0$, we see that there is a unique $\hat\alpha_2 \in C_p$ such that $p | \beta_0 \hat\alpha_2$. For $\alpha_2 \neq \hat\alpha_2$, we have $\beta_0 \alpha_2 \in S^{\rm prim}$. Let $\beta_0 \hat\alpha_2 = p \beta_0'$. Then $\beta_0' \in S^{\rm prim}$ and $p \not| \, |\beta_0'|^2$ (since we have assumed that $p^2 \not| \, |\beta_0|^2$). Hence, by Lemma \ref{fund-lemma}, we see that, for all $\alpha_1 \in C_p$, we have $\alpha_1^{-1} \beta_0 \hat\alpha_2 = \bar\alpha_1 \beta_0' \in S^{\rm prim}$. This implies $\nu_p(\alpha_1^{-1} \beta \hat\alpha_2) = s$ for all $\alpha_1 \in C_p$. If $\alpha_2 \neq \hat\alpha_2$, then Lemma \ref{fund-lemma} implies that there is a unique $\alpha_{1,2} \in C_p$ such that $\nu_p(\alpha_{1,2}^{-1} \beta \alpha_2) = s$. For all $\alpha_1 \neq \alpha_{1,2}$, we have $\nu_p(\alpha_1^{-1} \beta \alpha_2) = s-1$. This gives us
\begin{align*}
\sum\limits_{\alpha_1, \alpha_2 \in C_p} A(\alpha_1^{-1} \beta \alpha_2) &= \sum\limits_{\alpha_1 \in C_p} A(\alpha_1^{-1} \beta \hat\alpha_2) + \sum\limits_{\substack{\alpha_2 \in C_p \\ \alpha_2 \neq \hat\alpha_2}} \Big(A(\alpha_{1,2}^{-1} \beta \alpha_2) + \sum\limits_{\substack{\alpha_1 \in C_p \\ \alpha_1 \neq \alpha_{1,2}}} A(\alpha_1^{-1} \beta \alpha_2)\Big) \\
&= (p+1) A(\beta) + pA(\beta) + p^2 p^s |\beta_0| \sum\limits_{k=0}^{s-1} c\Big(\frac{-|\beta_0|^2p^{2s-2k}}2\Big).
\end{align*}
Putting this all together, we see that $p^2(A(p\beta) + A(p^{-1}\beta)) + p\sum_{\alpha_1, \alpha_2} A(\alpha_1^{-1} \beta \alpha_2)$ is equal to 
\begin{align*}
& p^2 p^s |\beta_0| \Big( (\lambda_p^2-2) \sum\limits_{k=0}^{s-1} c\Big(\frac{-|\beta_0|^2p^{2s-2k}}2\Big) + p c\Big(\frac{-|\beta_0|^2p^2}2\Big) + p c\Big(\frac{-|\beta_0|^2}2\Big)\Big) \\
& \qquad + p(p+1) A(\beta) + p^2A(\beta) + p^3 p^s |\beta_0| \sum\limits_{k=0}^{s-1} c\Big(\frac{-|\beta_0|^2p^{2s-2k}}2\Big)\\
=& p^2(\lambda_p^2-2) A(\beta) + p^2p^s |\beta_0| \Big(-(\lambda_p^2-2)c\Big(\frac{-|\beta_0|^2}2\Big) + (\lambda_p^2-2)c\Big(\frac{-|\beta_0|^2}2\Big) + pc\Big(\frac{-|\beta_0|^2}2\Big)\Big) \\
&\qquad + p(p+1) A(\beta) + p^2A(\beta) + p^3 A(\beta) - p^3 p^s |\beta_0| c\Big(\frac{-|\beta_0|^2}2\Big)\\
=& \big(p^2(\lambda_p^2-2) + p(p+1) + p^2 + p^3\big) A(\beta) \\
=& \big(p^2 \lambda_p^2 + p^3 + p\big) A(\beta).
\end{align*}
Here, we have used \eqref{lambda2-eq2} and $p | \,|\beta_0|^2$.

{\it Case 3}: Let $p^2 | \, |\beta_0|^2$. As in Case 2 above, Lemma \ref{fund-lemma} applied to $\beta_0$ implies that there is a unique $\hat\alpha_2 \in C_p$ such that $p | \beta_0 \hat\alpha_2$. For $\alpha_2 \neq \hat\alpha_2$, we have $\beta_0 \alpha_2 \in S^{\rm prim}$. Let $\beta_0 \hat\alpha_2 = p \beta_0'$. Then $\beta_0' \in S^{\rm prim}$ and $p | \, |\beta_0'|^2$ (since we have assumed $p^2 | \, |\beta_0|^2$). Hence by Lemma \ref{fund-lemma}, there is a unique $\hat\alpha_{1,2} \in C_p$ such that $\nu_p(\hat\alpha_{1,2}^{-1} \beta \hat\alpha_2) = s+1$, and for all $\alpha_1 \neq \hat\alpha_{1,2}$, we have $\nu_p(\alpha_1^{-1} \beta \hat\alpha_2) = s$. If $\alpha_2 \neq \hat\alpha_2$, then Lemma \ref{fund-lemma} implies that there is a unique $\alpha_{1,2} \in C_p$ such that $\nu_p(\alpha_{1,2}^{-1} \beta \alpha_2) = s$. For all $\alpha_1 \neq \alpha_{1,2}$, we have $\nu_p(\alpha_1^{-1} \beta \alpha_2) = s-1$. This gives us that $\sum A(\alpha_1^{-1} \beta \alpha_2)$ is equal to
\begin{align*}
& A(\hat\alpha_{1,2}^{-1} \beta \hat\alpha_2) + \sum\limits_{\alpha_1 \neq \hat\alpha_{1,2}} A(\alpha_1^{-1} \beta \hat\alpha_2) + \sum\limits_{\alpha_2 \neq \hat\alpha_2} \Big(A(\alpha_{1,2} \beta \alpha_2) + \sum\limits_{\alpha_1 \neq \alpha_{1,2}} A(\alpha_1^{-1} \beta \alpha_2)\Big) \\
=& p^s |\beta_0| \sum\limits_{k=0}^{s+1} c\Big(\frac{-|\beta_0|^2p^{2s-2k}}2\Big) + pA(\beta) + pA(\beta) + p^2 p^s |\beta_0| \sum\limits_{k=0}^{s-1} c\Big(\frac{-|\beta_0|^2p^{2s-2k}}2\Big) \\
=& A(\beta) + p^s |\beta_0| c\Big(\frac{-|\beta_0|^2}{2p^2}\Big) + pA(\beta) + pA(\beta) + p^2 A(\beta) - p^2 p^s |\beta_0| c\Big(\frac{-|\beta_0|^2}2\Big).
\end{align*}
Putting this all together, we see that $p^2(A(p\beta) + A(p^{-1}\beta)) + p\sum_{\alpha_1, \alpha_2} A(\alpha_1^{-1} \beta \alpha_2)$ is equal to 
\begin{align*}
&p^2(\lambda_p^2-2) A(\beta) + p^2p^s |\beta_0| \Big(-(\lambda_p^2-2)c\Big(\frac{-|\beta_0|^2}2\Big) + p c\Big(\frac{-|\beta_0|^2p^2}2\Big) + pc\Big(\frac{-|\beta_0|^2}2\Big)\Big) \\
& \qquad + p(1+2p+p^2) A(\beta) + p^2 p^s |\beta_0| \Big(p^{-1} c\Big(\frac{-|\beta_0|^2}{2p^2}\Big) - p c\Big(\frac{-|\beta_0|^2}2\Big)\Big) \\
=& \big(p^2 (\lambda_p^2-2) + p(1+2p+p^2)\big) A(\beta) \\
=& \big(p^2 \lambda_p^2 + p^3 + p\big) A(\beta).
\end{align*}
Here, we have used \eqref{lambda2-eq2} and $p^2 | \,|\beta_0|^2$. This completes the proof of the proposition.

\end{proof}

\section{The automorphic representation corresponding to the lifting}
In this section, we will use the Hecke equivariance from the previous section to determine the local components of the automorphic representation corresponding to the lifting. This will lead us to the conclusion that we have obtained a CAP representation and have found a couterexample of the Ramanujan conjecture.

\subsection{The local components of the automorphic representation}\label{localcomp-autorep}
Let $f \in S(\Gamma_0(2);-(\frac{1}{4}+\frac{r^2}{4}))$ be a  Hecke eigenform with Hecke eigenvalue $\lambda_p$ for every odd prime $p$. Let $F = F_f$ be as defined in Theorem \ref{lift-thm}. Let $\Phi_F : \cG(\A) \rightarrow \C$ be defined by 
\[
\Phi_F(\gamma g_{\infty}u_f)=F(g_{\infty})\quad\forall(\gamma,g_{\infty},u_f)\in\cG(\Q)\times\cG(\R)\times U.
\]
See Section \ref{adel-section} for details. Let $\pi_F$ be the irreducible cuspidal automorphic representation of $\cG(\A)$ generated by the right translates of $\Phi_F$. Note that the irreducibility follows from the strong multiplicity one result for $\cG(\A)$ (see \cite{Bad}, \cite{Bad-R}). The representation $\pi_F$ is cuspidal since $F$ is a cusp form. Let $\pi_F = \otimes'_p \pi_p$, where $\pi_p$ is an irreducible admissible representation of $\cG(\Q_p)$ for $p<\infty$ and $\pi_\infty$ is an irreducible admissible  representation of $\cG(\R)$. Recall $U = \prod_{p<\infty} K_p$ where $K_p$ is the maximal compact subgroup of $\cG_p$~(cf.~Section \ref{Autom-form}). Hence, for $p < \infty$, the representation $\pi_p$ is a spherical representation and can be realized as a subrepresentation of an unramified principal series representation, i.e. a representation induced from an unramified character of the Borel subgroup. The representation $\pi_p$ is completely determined by the action of the Hecke algebra $H(\cG_p, K_p)$ on the spherical vector in $\pi_p$, which in turn, is completely determined by the Hecke eigenvalues of $F$ obtained in the previous section. See \cite{Ca} for details. 
For $p=2$ we need to assume that $f$ is a new form for the determination of Hecke eigenvalue of $F_f$~(cf.~Section \ref{Hecke-equiv-even}). 

\subsubsection*{Description of $\pi_p$ for $p$ odd}
If $p$ is an odd prime, then we have $\cG_p = \GL_4(\Q_p)$ and $K_p = \GL_4(\Z_p)$. Given $4$ unramified characters $\chi_1, \chi_2, \chi_3, \chi_4$ of $\Q_p^\times$, we obtain a character $\chi$ of the Borel subgroup $P$ of upper triangular matrices in $\cG$, by
\begin{equation}\label{chi-p-odd}
\chi(\begin{bmatrix}a_1& \ast & \ast & \ast\\ & a_2 & \ast & \ast \\ && a_3 & \ast \\ &&&a_4 \end{bmatrix}) = \chi_1(a_1) \chi_2(a_2) \chi_3(a_3) \chi_4(a_4).
\end{equation}
The modulus character $\delta_P$ is given by
\begin{equation}\label{modulus-p-odd}
\delta_P(\begin{bmatrix}a_1& \ast & \ast & \ast\\ & a_2 & \ast & \ast \\ && a_3 & \ast \\ &&&a_4 \end{bmatrix}) = |a_1^3a_2a_3^{-1}a_4^{-3}|,
\end{equation}
where $|*|$ denotes the $p$-adic absolute value. 
The unramified principal representation corresponding to $\chi$ is given by $I(\chi)$ which consists of locally constant functions $f : \GL_4(\Q_p) \rightarrow \C$, satisfying
$$f(bg) = \delta_P(b)^{1/2} \chi(b) f(g), \text{ for all } b \in P, g \in \GL_4(\Z_p).$$
The action of the Hecke algebra is as follows. If $\phi \in H(\GL_4(\Q_p), \GL_4(\Z_p))$ and $f \in I(\chi)$, define
\begin{equation}\label{hecke-action-p-odd}
\big(\phi \ast f\big)(g) = \int\limits_{\GL_4(\Q_p)} \phi(h) f(gh) dh.
\end{equation}
Recall that we have normalized the measure $dh$ on $\GL_4(\Q_p)$ so that the volume of $\GL_4(\Z_p)$ is $1$. Let $f = f_0$, the unique vector in $I(\chi)$ that is right invariant under $\GL_4(\Z_p)$ and $f_0(1)=1$, and $\phi = \phi_h$ a characteristic function of $\GL_4(\Z_p) h \GL_4(\Z_p) = \sqcup_i h_i \GL_4(\Z_p)$. It follows from \eqref{hecke-action-p-odd} that
\begin{equation}\label{hecke-spherical-vector-p-odd}
\big(\phi_h \ast f_0\big)(1) = \sum\limits_i f_0(h_i) = \mu_h,
\end{equation}
where $\mu_h$ is determined by the representation $\pi_p$. The Hecke algebra $H(\GL_4(\Q_p), \GL_4(\Z_p))$ is generated by $\{\phi_1^{\pm 1}, \phi_2, \phi_3, \phi_4\}$, defined in \eqref{odd-hecke-ops}.
\begin{lem}\label{local-hecke-evals-p-odd}
Let $\mu_1, \mu_2, \mu_3, \mu_4$ be the constants obtained by the action of $\phi_i, i = 1,2,3,4$ on the spherical vector $f_0$ in $\pi_p$ according to \eqref{hecke-spherical-vector-p-odd}. Then $\mu_1 = 1, \mu_2 = \mu_4 = p(p+1) \lambda_p$ and $\mu_3 = p^2 \lambda_p^2 + p^3 + p$.
\end{lem}
\begin{proof}
The lemma follows from the fact that the action of the $p$-adic Hecke algebra on the spherical vector in $\pi_p$ is exactly the same as the action of the $p$-part of the classical Hecke algebra on $F$. 
First note that $\phi_1$ acts as the identity operator, which implies $\mu_1=1$. 
The other Hecke eigenvalues follow from Propositions \ref{oddp-hecke-1-2-prop} and \ref{oddp-hecke-3-prop}.
\end{proof}

Recall that Proposition \ref{Gl4-double-cos-lem} gives the double coset decompositions which can be used to determine the action of $\phi$ on $f_0$. Let us abbreviate $\alpha_i = \chi_i(p)$ for $i = 1,2,3,4$. Working in the induced model $I(\chi)$ of $\pi_p$, we see that 
\begin{align}\label{hecke-action-alpha-p-odd}
\big(\phi_1 \ast f_0)(1) &= \alpha_1 \alpha_2 \alpha_3 \alpha_4, \\ \nonumber
\big(\phi_2 \ast f_0)(1) &= p^3 p^{-3/2} \alpha_1 \alpha_2 \alpha_3 + p p^{1/2} \alpha_1 \alpha_3 \alpha_4 + p^2 p^{-1/2} \alpha_1 \alpha_2 \alpha_4 + p^{3/2} \alpha_2 \alpha_3 \alpha_4 \\ \nonumber
&= p^{3/2} \alpha_1 \alpha_2 \alpha_3 \alpha_4 \Big(\alpha_1^{-1} + \alpha_2^{-1} + \alpha_3^{-1} + \alpha_4^{-1}\Big), \\ \nonumber
\big(\phi_4 \ast f_0)(1) &= p^3 p^{-3/2} \alpha_1 + p p^{1/2} \alpha_3 + p^2 p^{-1/2} \alpha_2 + p^{3/2} \alpha_4 \\ \nonumber
&= p^{3/2} \Big( \alpha_1 + \alpha_2 + \alpha_3 + \alpha_4\Big), \\ \nonumber
\big(\phi_3 \ast f_0)(1) &= p^4 p^{-2} \alpha_1 \alpha_2 + p^2 \alpha_2 \alpha_3 + p p \alpha_2 \alpha_4 + p^3 p^{-1} \alpha_1 \alpha_3 + p^2 \alpha_1 \alpha_4 + p^2 \alpha_3 \alpha_4 \\ \nonumber
&= p^2 \Big(\alpha_1 \alpha_2 + \alpha_2 \alpha_3 + \alpha_2 \alpha_4 + \alpha_1 \alpha_3 + \alpha_1 \alpha_4 + \alpha_3 \alpha_4 \Big). 
\end{align}

\begin{prop}\label{local-rep-p-odd-prop}
Let $f \in S(\Gamma_0(2);-(\frac{1}{4}+\frac{r^2}{4}))$ be a  Hecke eigenform with Hecke eigenvalue $\lambda_p$ for every odd prime $p$. Let $F = F_f$ be as defined in Theorem \ref{lift-thm}. Let $\pi_F = \otimes'_p \pi_p$ be the corresponding irreducible cuspidal automorphic representation of $\cG(\A)$. For an odd prime $p$, the representation $\pi_p$ is the unique spherical constituent of the unramified principal series representation $I(\chi)$ where, up to the action of the Weyl group of $\GL_4$, the character $\chi$ is given by 
\begin{align}\label{alpha-values-p-odd}
\chi_1(p) = p^{1/2} \frac{\lambda_p+\sqrt{\lambda_p^2-4}}2, & \quad \chi_2(p) = p^{1/2} \frac{\lambda_p-\sqrt{\lambda_p^2-4}}2, \\ \nonumber
 \chi_3(p) = p^{-1/2} \frac{\lambda_p+\sqrt{\lambda_p^2-4}}2, & \quad \chi_4(p) = p^{-1/2} \frac{\lambda_p-\sqrt{\lambda_p^2-4}}2.  
 \end{align}
\end{prop} 
\begin{proof}
The representation $I(\chi)$ corresponding to $\pi_p$ is generated by the spherical vector $f_0$ and hence, it is completely determined by the action of the generators of the Hecke algebra on $f_0$. The representation $\pi_p$ is also determined by the Hecke eigenvalues of $F$ under the $p$-part of the classical Hecke algebra. Substituting the values of $\alpha_i = \chi_i(p), i =1,2,3,4$ from \eqref{alpha-values-p-odd}, into \eqref{hecke-action-alpha-p-odd} shows that we get the exact same eigenvalues as in Lemma \ref{local-hecke-evals-p-odd}. This completes the proof of the proposition.
\end{proof}

Let us remark here that we can use Lemma \ref{local-hecke-evals-p-odd} to directly solve for $\alpha_i$ from \eqref{hecke-action-alpha-p-odd}. It is a tedious computation but results in the same answer as in the statement of the above proposition. 

\subsubsection*{Description of $\pi_2$}
Recall that $B_2 = B \otimes_\Q \Q_2$, where $B$ is a definite quaternion algebra over $\Q$ with discriminant $2$ and $\mO_2$ is the completion of the Hurwitz order $\mO$ at $2$. In this case $\cG_2 = \GL_2(B_2)$ and $K_2 = \GL_2(\mO_2)$. Given two unramified characters $\chi_1, \chi_2$ of $B_2^\times$, we obtain a character $\chi$ of the Borel subgroup of upper triangular matrices on $\cG$ by
$$\chi(\mat{\alpha}{\ast}{0}{\beta}) = \chi_1(\alpha) \chi_2(\beta).$$
The modulus character is given by 
$$\delta(\mat{\alpha}{\ast}{0}{\beta}) = |\alpha/\beta|^2.$$
Here, $| \, |$ is the $2$-adic absolute value of the reduced norm of $B_2$. The unramified principal series representation corresponding to $\chi$ is given by $I(\chi)$ which consists of locally constant functions $f : \cG_2 \rightarrow \C$, satisfying
$$f(bg) = \delta(b)^{1/2} \chi(b) f(g), \text{ for all } b \in \text{ Borel subgroup, }  g \in \cG_2.$$
The action of the Hecke algebra is as follows. If $\phi \in H(\cG_2,K_2)$ and $f \in I(\chi)$, define
\begin{equation}\label{hecke-action-p-2}
\big(\phi \ast f\big)(g) = \int\limits_{\cG_2} \phi(h) f(gh) dh.
\end{equation}
Recall that we have normalized the measure $dh$ on $\cG_2$ so that the volume of $K_2$ is $1$. Let $f = f_0$, the unique vector in $I(\chi)$ that is right invariant under $K_2$ and $f_0(1) = 1$, and $\phi = \phi_h$ a characteristic function of $K_2 h K_2 = \sqcup_i h_i K_2$. It follows from \eqref{hecke-action-p-odd} that
\begin{equation}\label{hecke-spherical-vector-p-2}
\big(\phi_h \ast f_0\big)(1) = \sum\limits_i f_0(h_i) = \mu_h,
\end{equation}
where $\mu_h$ is determined by the representation $\pi_2$. The Hecke algebra $H(\cG_2, K_2)$ is generated by $\{\varphi_1^{\pm 1}, \varphi_2\}$, where $\varphi_1,~\varphi_2$ denote the characteristic functions for
\[
K_2
\begin{bmatrix}
\varpi_2 & 0\\
0 & \varpi_2
\end{bmatrix}K_2, \quad K_2
\begin{bmatrix}
\varpi_2 & 0\\
0 & 1
\end{bmatrix} K_2
\]
Here $\varpi_2$ is a uniformizer for $B_2$.
\begin{lem}\label{local-hecke-evals-p-2}
Let $\mu_1, \mu_2$ be the constants obtained by the action of $\phi_i, i = 1,2$ on the spherical vector $f_0$ in $\pi_2$ according to \eqref{hecke-spherical-vector-p-2}. Then $\mu_1 = 1$ and $\mu_2 = -3\sqrt{2} \epsilon$, where $\epsilon$ is the Atkin Lehner eigenvalue of $f$.
\end{lem}
\begin{proof}
The proof is the same as in the case of an odd prime.
\end{proof}

Recall that Proposition \ref{Gl4-double-cos-lem} gives the double coset decompositions which can be used to determine the action of $\phi$ on $f_0$. Let us abbreviate $\alpha_i = \chi_i(\varpi_2)$ for $i = 1, 2$. Working in the induced model $I(\chi)$ of $\pi_2$, we see that 
\begin{align}\label{hecke-action-alpha-p-2}
(\varphi_1 \ast f_0)(1) &= \alpha_1 \alpha_2, \\
(\varphi_2 \ast f_0)(1) &= 2(\alpha_1 + \alpha_2). \nonumber
\end{align}
\begin{prop}\label{local-rep-p-2-prop}
Let $f \in S(\Gamma_0(2);-(\frac{1}{4}+\frac{r^2}{4}))$ be a new form with Hecke eigenvalue $\lambda_p$ for $p=2$ and Atkin Lehner eigenvalue $\epsilon$, for which $\lambda_2=-\epsilon$ holds~(cf.~(\ref{lambda2-epsilon-reln})). Let $F = F_f$ be as defined in Theorem \ref{lift-thm}. Let $\pi_F = \otimes'_p \pi_p$ be the corresponding irreducible cuspidal automorphic representation of $\cG(\A)$. The representation $\pi_2$ is the unique spherical constituent of the unramified principal series representation $I(\chi)$ where, up to the action of the Weyl group, the character $\chi$ is given by 
\begin{equation}\label{alpha-values-p-2}
\chi_1(\varpi_2) = -\sqrt{2} \epsilon, \qquad  \chi_2(\varpi_2) = -1/\sqrt{2} \epsilon. 
 \end{equation}
\end{prop} 
\begin{proof}
The proof is the same as in the case of odd prime.
\end{proof}
\subsubsection*{Description of $\pi_{\infty}$}
Let us note that $F=F_f\in {\cal M}(\GL_2(\mO);r)$ implies that the archimedean component $\pi_{\infty}$ of $\pi_F$ is spherical. 
Namely, up to constant multiples, $\pi_{\infty}$ has a unique $K_{\infty}$-invariant vector, where we put $K_{\infty}:=K$ with $K$ as in \eqref{Iwasawa-decomp-alg}. 

We now introduce $M_{\infty}:=
\left\{\left.
\begin{pmatrix}
u_1 & 0\\
0 & u_2
\end{pmatrix}~\right|~u_1,u_2\in\Hl^1\right\}$, where see Section \ref{Lie-alg} for  $\Hl^1$. 
Let $P_{\infty}$ be the standard proper parabolic subgroup $\cG_{\infty}=GL_2(\Hl)$ given by 
\[
\left\{
\begin{pmatrix}
a & *\\
0 & d
\end{pmatrix}\in\cG_{\infty}\right\}.
\]
We have $P_{\infty}:=Z^+NAM_\infty$, where $Z^+,~N$ and $A$ are as in \eqref{Iwasawa-decomp-alg}. 
The group $Z^+AM_{\infty}$ is nothing but the Levi subgroup of $P_{\infty}$. We now note that the Langlands classification of real reductive groups~(cf.~\cite{La}) implies that $\pi_{\infty}$ has to be embedded into some principal series representation $I_{P_{\infty}}$ of $\cG_{\infty}$ induced from a quasi-character of $P_{\infty}$. 
Since $\pi_{\infty}$ is spherical $I_{P_{\infty}}$ is also spherical. Namely $I_{P_{\infty}}$ has a unique $K_{\infty}$-invariant vector, up to constant multiples. As $\pi_{\infty}$ has the trivial central character, so does $I_{P_{\infty}}$. These imply that the quasi-character of $P_{\infty}$ inducing $I_{P_{\infty}}$ has to be trivial on $Z^+M_{\infty}$. 
For $s\in\C$ we introduce the quasi-character $\chi_s$ of $P_{\infty}$ defined by
\[
\chi_s\left(
\begin{pmatrix}
a & *\\
0 & d
\end{pmatrix}\right)=\nu(ad^{-1})^s, 
\]
where recall that $\nu$ denotes the reduced norm of $\Hl$~(cf.~Section \ref{gps-hypsp}). 
For this we note that $\chi_s$ is trivial on $Z^+M_{\infty}$. 
We furthermore introduce the modulus character $\delta_{\infty}$ of $P_{\infty}$. 
The principal series representation $I_{P_{\infty}}$ is thus expressed as 
\[
I_{P_{\infty}}={\rm Ind}_{P_{\infty}}^{\cG_{\infty}}(\delta_{\infty}\chi_{s}).
\]
\begin{prop}\label{local-rep-infty-prop}
We have an isomorphism
\[
\pi_{\infty}\simeq{\rm Ind}_{P_{\infty}}^{\cG_{\infty}}(\delta_{\infty}\chi_{\pm\sqrt{-1}r})
\]
as $(\g, K_{\infty})$-modules, where recall that $\g$ denotes the Lie algebra of $\cG_{\infty}$~(cf.~Section \ref{Lie-alg}).
\end{prop}
\begin{proof}
Let $v$ be a unique $K_{\infty}$-invariant vector in the representation space of ${\rm Ind}_{P_{\infty}}^{\cG_{\infty}}(\delta_{\infty}\chi_{s})$, which $\pi_{\infty}$ can be embedded into. Then $v$ can be also regarded as a vector of $\pi_{\infty}$. 
We remark that $\pi_{\infty}$ can be viewed as a representation of $SL_2(\Hl)\simeq GL_2(\Hl)/Z^+$~(cf.~Section \ref{Lie-alg}) since it  has the trivial central character. 
Consider the infinitesimal action of the Casimir operator $\Omega$~(cf.~\eqref{Casimir-defn}) on $v$. We then have
\[
\Omega\cdot v=\left(\frac{s^2}{4}-1\right)v=\left(-\frac{r^2}{4}-1\right)v,
\]
which leads to $s=\pm\sqrt{-1}r$. Now recall that we have assumed $r\in\R$~(cf.~Section \ref{Autom-form}). We thus know that the quasi-character $\chi_s$ is parametrized by a purely imaginary number $\pm\sqrt{-1}r$. By Harish-Chandra \cite[Section 41,~Theorem 1]{Ha} the spherical principal series representation ${\rm Ind}_{P_{\infty}}^{\cG_{\infty}}(\delta_{\infty}\chi_{\pm\sqrt{-1} r})$ is an irreducible unitary representation. 
For this see also \cite[Remark (2.1.13)]{Co} and note the accidental isomorphism $Spin(5,1)\simeq SL_2(\Hl)$ as real Lie groups. 
Consequently we have the isomorphism in the assertion.
\end{proof}
\subsection{CAP representations}
Let us first give the definition of CAP representations. 
\begin{defn}\label{CAP-def}
Let $G_1$ and $G_2$ be two reductive algebraic groups over a number field such that $G_{1,v} \simeq G_{2,v}$ for almost all places $v$. Let $P_2$ be a parabolic subgroup of $G_2$ with Levi decomposition $P_2 = M_2 N_2$. An irreducible cuspidal automorphic representation $\pi = \otimes'_v \pi_v$ of $G_1(\A)$ is called {\it cuspidal associated to parabolic} (CAP) $P_2$, if there exists an irreducible cuspidal automorphic representation $\sigma$ of $M_2$ such that $\pi_v \simeq \pi_v'$ for almost all places $v$, where $\pi' = \otimes'_v \pi_v'$ is an irreducible component of ${\rm Ind}_{P_2(\A)}^{G_2(\A)}(\sigma)$.
\end{defn}

See \cite{G1} and \cite{P} for details on CAP representations defined for two groups instead of just one. Take $G_1 = \cG = \GL_2(B)$ and $G_2 = \GL_4$. Here $B$ is a definite quaternion algebra with discriminant $2$. Since these groups are inner forms of each other, we have $G_{1,p} \simeq G_{2,p}$ for all odd primes $p$. Let $P_2$ be the standard parabolic of $\GL_4$ with Levi subgroup $M_2 = \GL_2 \times \GL_2$. Let $f \in S(\Gamma_0(2);-(\frac{1}{4}+\frac{r^2}{4}))$ be a  Hecke eigenform with Hecke eigenvalue $\lambda_p$ for every odd prime $p$ and Atkin Lehner eigenvalue $\epsilon$. Let $\sigma = \otimes'_p \sigma_p$ be the irreducible cuspidal automorphic representation of $\GL_2$ corresponding to $f$. For an odd prime $p$, the representation $\sigma_p$ is the spherical principal series representation $I(\eta)$, where $\eta$ is given by 
$$\eta(\mat{a}{b}{}{d}) = \eta_0(a) \eta_0^{-1}(d).$$
Here, $\eta_0$ is an unramified character of $\Q_p^\times$ such that $\eta_0(p) + \eta_0^{-1}(p) = \lambda_p$. For $p=2$ assume that $f$ is a new form. Then the representation $\sigma_2$ is the twist of the Steinberg representation of $\GL_2(\Q_2)$ by an unramified character $\eta'$, with $\eta'(2) = -\epsilon$. The representation $\sigma$ gives a representation $|{\rm det}|^{-1/2}\sigma \times |{\rm det}|^{1/2}\sigma$ of $M_2$. We have the following theorem.
\begin{thm}\label{CAP-thm}
Let $f \in S(\Gamma_0(2);-(\frac{1}{4}+\frac{r^2}{4}))$ be a  Hecke eigenform with Hecke eigenvalue $\lambda_p$ for every odd prime $p$ and Atkin Lehner eigenvalue $\epsilon$. Let $\sigma = \otimes'_p \sigma_p$ be the irreducible cuspidal automorphic representation of $\GL_2$ corresponding to $f$. Let $F = F_f$ be as defined in Theorem \ref{lift-thm}. Let $\pi_F = \otimes'_p \pi_p$ be the corresponding irreducible cuspidal automorphic representation of $\cG(\A)$. Then $\pi_F$ is CAP to an irreducible component of ${\rm Ind}_{P_2(\A)}^{G_2(\A)}(|{\rm det}|^{-1/2}\sigma \times |{\rm det}|^{1/2}\sigma)$.
\end{thm}
\begin{proof}
The theorem follows from the observation that, for an odd prime $p$, we have the isomorphism ${\rm Ind}_{P_2(\Q_p)}^{G_2(\Q_p)}(|{\rm det}|_p^{-1/2}\sigma_p \times |{\rm det}|_p^{1/2}\sigma_p) \simeq I(\chi_p)$. Here, $I(\chi_p)$ is the representation described in Proposition \ref{local-rep-p-odd-prop}. A concrete map is given as follows. For $f \in {\rm Ind}_{P_2(\Q_p)}^{G_2(\Q_p)}(|{\rm det}|_p^{-1/2}\sigma_p \times |{\rm det}|_p^{1/2}|\sigma_p)$ define the function $g \mapsto (f(g))(I_2, I_2)$. Note that $\delta_{P_2}({\rm diag}(a_1, a_2, a_3, a_4)) = |a_1a_2a_3^{-1}a_4^{-1}|^2$. 
\end{proof}
We can furthermore show that our cuspidal representations $\pi_F$'s provide counterexamples of the Ramanujan conjecture. 
\begin{thm}\label{Counter-eg-RC}
Let $\pi_F=\otimes'_p\pi_p$ be as in Theorem \ref{CAP-thm}.  For every odd prime $p$~(respectively~$p=\infty$), $\pi_p$ is non-tempered~(respectively~tempered). 
If we further assume that $f$ is a new form, $\pi_p$ is non-tempered for every finite prime $p$ and tempered for $p=\infty$.
\end{thm}
\begin{proof}
The temperedness of $\pi_{\infty}$ is due to Proposition \ref{local-rep-infty-prop} and \cite[Remark 2.1.13]{Co}. 
For an odd prime $p$, the unramified chacters $\chi_i$ with $1\le i\le 4$ are not unitary~(cf.~\eqref{alpha-values-p-odd}). This means that $\pi_p$ is non-tempered~(cf.~\cite{Sk}).

Let $p=2$ and suppose that $f$ is a new form. We recall that $f_0$ denotes the spherical vector in $\pi_2$, and introduce its dual vector $f'_0$ in the contragredient representation of $\pi_2$. 
With the invariant measure $dg$ of $\cG_p/Z_p$ normalized so that $\displaystyle\int_{K_2/Z_p}dg=1$, for any $\delta>0$, we consider the following integral of the matrix coefficient
\[
\displaystyle\int_{\cG_2/Z_2}|\langle\pi_2(g)f_0,f'_0\rangle|^{2+\delta}dg
\]
over $\cG_2$ modulo center $Z_2$, 
where $\langle*,*\rangle$ denotes the canonical paring of $\pi_2$ and its contradredient.  If $\pi_2$ is tempered, this integral should be convergent. 
Now we note that the set $(\bigsqcup_{n\ge 0}K_2
\begin{pmatrix}
\varpi_2^n & 0\\
0 & 1
\end{pmatrix}K_2)/Z_2$ can be regarded as a subdomain of $\cG_2/Z_2$ and that there is a decomposition
\[
K_2
\begin{pmatrix}
\varpi_2^n & 0\\
0 & 1
\end{pmatrix}K_2=
\underset{x\in\mO_2/\varpi_2^n\mO_2}{\sqcup}\begin{pmatrix}
\varpi_2^n & 0\\
0 & 1
\end{pmatrix}
\begin{pmatrix}
1 & \varpi_2^{-n}x\\
0 & 1
\end{pmatrix}K_p\sqcup
\begin{pmatrix}
1 & 0\\
0 & \varpi_2^n
\end{pmatrix}K_p.
\]
It is verified that the Hecke operator defined by $K_2
\begin{pmatrix}
\varpi_2^n & 0\\
0 & 1
\end{pmatrix}K_2$ acts on $f_0$ as follows:
\[
(K_2
\begin{pmatrix}
\varpi_2^n & 0\\
0 & 1
\end{pmatrix}K_2)\cdot f_0=(-\epsilon)^n(2^{3n/2}+2^{n/2})f_0.
\]
We thereby have a divergent integral
\[
\displaystyle\int_{(\bigsqcup_{n\ge 0}K_2
\begin{pmatrix}
\varpi_2^n & 0\\
0 & 1
\end{pmatrix}K_2)/Z_p}|\langle\pi_2(g)f_0,f'_0\rangle|^{2+\delta}dg=
(\displaystyle\sum_{n\ge 0}(2^{3n/2}+2^{n/2})^{2+\delta})|\langle f_0,f'_0\rangle|^{2+\delta}=\infty,
\]
which leads to a contradiction. We therefore see that $\pi_2$ is non-tempered. 
As a result we are done.
\end{proof}
\begin{rem}\label{local-rep-rem}
\begin{enumerate}
\item According to Tadi{\'c} \cite{Ta} the parabolic induction $I(\chi)$ for $p=2$~(cf.~Section \ref{localcomp-autorep}) has two composition factor, one of which is a unique essentially square integrable subquotient. Our non-tempered representation $\pi_2$ is the remaining non-square integrable composition factor. Besides our approach there seem several ways to prove that the non-square integrable composition factor is non-tempered. In fact, Marko Tadi{\'c} pointed out that the non-temperedness is proved by using the classification of the non-unitary dual of $GL(n)$ over a division algebra~(cf.~\cite{Ta}) or by Casselman's criterion on the temperedness of an  admissible representation.\\

\item From Wayl's law~(cf.~\cite[(11.5)]{Iwn}) we can deduce that there exist non-zero newforms in $S(\Gamma_0(2);-(\frac{r^2}{4}+\frac{1}{4}))$ for some $r\in\R$. Let $N_{\Gamma}(T)$ be the counting function of an orthogonal basis of the discrete spectrum for a congruence subgroup $\Gamma$ as in \cite[Section 11]{Iwn}. 
Put $N^*_{\Gamma_0(2)}(T)$ to be such counting function for newforms of $\Gamma_0(2)$. With the help of Casselman's local theory of oldforms and newforms~(cf.~\cite{Cas}) we deduce
\begin{align*}
N^*_{\Gamma_0(2)}(T) & =\frac{{\rm Vol}({\frak h}/\Gamma_0(2))-2{\rm Vol}({\frak h}/SL_2(\Z))}{4\pi}T^2+O(T\log T) \\
&=\frac{{\rm Vol}({\frak h}/SL_2(\Z))}{4\pi}T^2+O(T\log T)
\end{align*}
from Weyl's law just mentioned. This leads to the existence of a non-zero cuspidal representation $\pi_F$ whose local component $\pi_p$ is non-tempered at every $p<\infty$.
\end{enumerate}
\end{rem}

Masanori Muto\\
Kumamoto Prefectural Toryo High School\\
5-10, Komine 4-chome, Higashi-ku, Kumamoto 862-0933, Japan\\
{\it E-mail address}:~muto-m@mail.bears.ed.jp
\\[10pt]
Hiro-aki Narita\\
Graduate School of Science and Technology\\
Kumamoto University\\
Kurokami, Chuo-ku, Kumamoto 860-8555, Japan\\ 
{\it E-mail address}:~narita@sci.kumamoto-u.ac.jp
\\[10pt]
Ameya Pitale\\
Department of Mathematics\\
University of Oklahoma\\
Norman, Oklahoma, USA.\\ 
{\it E-mail address}:~apitale@ou.edu
\end{document}